\documentclass{amsart}
\usepackage{amssymb}
\usepackage{upref}
\usepackage{amsmath}
\usepackage[shortlabels]{enumitem}
\usepackage{tikz}
\usepackage{lipsum}

\makeatletter

\newcommand{\tpitchfork}{%
  \vbox{
    \baselineskip\z@skip
    \lineskip-.52ex
    \lineskiplimit\maxdimen
    \m@th
    \ialign{##\crcr\hidewidth\smash{$-$}\hidewidth\crcr$\pitchfork$\crcr}
  }%
}
\makeatother
\usepackage{epsfig}
\usepackage{amsopn}
\usepackage{amssymb, amscd, amsmath}
\usepackage{array}
\usepackage{color} 
\usepackage{enumerate}

\newtheorem{mainth}{Theorem}[]

\theoremstyle{plain}

\newtheorem{thm}{Theorem}[section]

\newtheorem*{thm*}{Theorem} 
\newtheorem{cor}[thm]{Corollary}
\newtheorem{lem}[thm]{Lemma}

\newtheorem{prop}[thm]{Proposition}

\theoremstyle{definition}
\newtheorem{question}{Question}
\newtheorem{defin}[thm]{Definition}
\newtheorem*{rems*}{Remarks}
\theoremstyle{remark}
\newtheorem{rem}[thm]{Remark}

\DeclareMathOperator{\Sym}{Sym}
\DeclareMathOperator{\Sing}{Sing_0}
\DeclareMathOperator{\Reg}{Reg}
\DeclareMathOperator{\Str}{Str}
\newcommand{\Gen}{\mathcal{G}_{g_0}}
 \DeclareMathOperator{\Scew}{Scew}

\DeclareMathOperator{\SO}{SO}

\DeclareMathOperator{\Rc}{Rm}

\newcommand{\so}{\mathfrak{so}}

\DeclareMathOperator{\codim}{codim}

\newcommand{\RR}{\mathbb{R}}
\newcommand{\RP}{\mathbb{RP}}
\newcommand{\NN}{\mathbb{N}}

\newcommand{\m}{\mathfrak{m}}
\newcommand{\bF}{\bar F}   
\newcommand{\bS}{\bar S}
\newcommand{\gijk}{g_0^{ijk}}
\newcommand{\gkij}{g_0^{kij}}

 \newcommand{\gijj}{g_0^{ijj}}
 
   \newcommand{\gikk}{g_0^{ikk}}
 
\newcommand{\cl}{\mathrm{cl}}

\usepackage[utf8]{inputenc}

\usepackage{xcolor}
\definecolor{Red}{rgb}{1.,0.,0.}

\DeclareMathOperator{\Met}{{\mathcal M}}
\providecommand{\Metmu}{{{\mathcal N}_\mu}}
\providecommand{\Yod}{{{\mathcal Y}_{g_0}}}
\DeclareMathOperator{\tr}{tr}
\DeclareMathOperator{\End}{End}

\providecommand{\integral}[4]{\int_{#1}^{#2} #3 \, #4}

\begin{document}

%
\begin{titlepage}\title
{Scalar curvature along Ebin geodesics}
\author{Christoph B\"ohm}\author{Timothy Buttsworth}\author{Brian Clarke}
\begin{abstract}
Let $M$ be a smooth, compact manifold and let
$\Metmu$ denote the set of Riemannian metrics on $M$
with smooth volume density $\mu$. 
For a given $g_0\in \Metmu$, we show that if $\dim(M)\ge 5$, then there exists 
an open and dense subset
$\Yod \subset T_{g_0} \Metmu$ (in the $C^{\infty}$ topology)
so that for each $h\in \Yod$, the $(\Metmu,L^2)$ Ebin geodesic $\gamma_h(t)$ with $\gamma_h(0)=g_0$ and $\gamma_h'(0)=h$ satisfies $\lim_{t \to \infty}$
 $R(\gamma_h(t))=-\infty$, uniformly. 
\end{abstract}
\end{titlepage}
\maketitle
\tableofcontents
\setcounter{page}{1}
\setcounter{tocdepth}{0}


\section{Introduction}
\label{sec:introduction}

Every compact manifold $M$
of dimension $n\geq 3$ admits Riemannian metrics $g$ of negative scalar curvature $R(g)<0$ \cite{Aub}
and even metrics of negative Ricci curvature \cite{Lo1,Lo2}.  
The main result of this paper, Theorem \ref{thm0}, 
shows in an algebraically controlled manner that there are many Riemannian metrics \textit{with the same volume form} that have \textit{arbitrarily large negative scalar curvature}.

Let  $\Metmu$ denote the space
of smooth Riemannian metrics on a compact Riemannian manifold $M$
with given $C^\infty$-smooth volume density $\mu:M \to (0,\infty)$ 
of volume one. Recall that
by \cite{Msr} all Riemannian structures on $M$  of volume one
can be realized by Riemannian metrics in  $\Metmu$. 
The space $\Metmu$
is an infinite-dimensional Fréchet manifold \cite{Eb}, which
can be endowed with the Ebin metric, also called the
$L^2$-metric. 
The Ebin geodesics $\gamma_h(t)$ of $(\Metmu,L^2)$
emanating from an arbitrary point 
$g_0\in \Metmu$ with $\gamma_h'(0)=h \in T_{g_0}\Metmu$  are given by
$\gamma_h(t)(\,\cdot \,,\cdot )=g_0(\exp(tH)\,\cdot ,\cdot )$,  where $h(\cdot,\cdot )=g_0(H \,\cdot ,\cdot)$ (see Section \ref{preliminaries}).
Our main result is as follows:
\begin{mainth}\label{thm0}
 Let $(M,g_0)$ be a compact Riemannian manifold with volume density $\mu$.
 Then, if $\dim(M)\geq 5$, there exists an open and dense
 subset $\Yod \subset T_{g_0}\Metmu$ in the $C^{\infty}$ topology
 such that for every $h \in \Yod$ we have 
\begin{equation}
    \lim_{t\to \infty}R(\gamma_h(t))=-\infty 
\end{equation}
uniformly on $M$.
\end{mainth}

Theorem \ref{thm0} says that for \emph{any} base point
$g_0\in \Metmu$, there is an open and dense set of initial directions for which
the scalar curvature tends to $-\infty$ 
along the corresponding
geodesic ray as we approach the boundary at infinity of $\Metmu$. 
Moreover, we would like to mention that
the scalar curvature tends to $-\infty$  exponentially fast: for each $h\in \mathcal{Y}_{g_0}$, there exist positive constants $C_1(h),C_2(h),T(h)$ so that, for all $t\geq T(h)$, we have 
\begin{align*}
    R(\gamma_h(t))\le -C_1(h)\cdot  e^{C_2(h)\cdot  t}.
\end{align*}
The constants $C_1,C_2,T$ can be chosen to depend continuously on $h$ in the $C^2$ topology.

Due to the Theorem of Gauss-Bonnet, Theorem
\ref{thm0} cannot hold for dimension $n=2$. In dimension $n=3$, we show 
in Lemma \ref{lem:3dplus}
that for \emph{each} $g_0 \in \Metmu$, there cannot exist a dense set $\mathcal{Y}_{g_0} \subset T_{g_0}\Metmu$
with the above properties.
The case $n=4$ is wide open: see Remark \ref{rem:dim4}.

It is clear that the open and dense set $\Yod$ cannot be chosen to be \textit{all} of $T_{g_0}\mathcal{N}_{\mu}$ for 
arbitrary $g_0 \in \Metmu$ on arbitrary $M$.
Indeed, if $g_0$ is a Riemannian submersion metric  and if
the induced metric on all fibres has positive scalar curvature, then the canonical variation gives rise
to a choice of $h$ for which the scalar curvature tends to $+\infty$ uniformly
along $\gamma_h(t)$ \cite[9.67]{Bes}.
 In general we expect scalar curvature to tend to $+\infty$ only in directions whose eigenspace 
 distribution is related to certain singular foliations (in the sense of Sussmann \cite{Su1}, \cite{Su2}).

\begin{question}
 Let $(M,g_0)$ be a compact Riemannian manifold with volume density $\mu$.
Let $h \in  T_{g_0}\Metmu$ and suppose that 
$ \lim_{t\to \infty}R(\gamma_h(t))=+\infty$
uniformly on $M$. Are there integrability conditions that $h$ must satisfy?
\end{question}

We turn to the proof of Theorem \ref{thm0}. 
Given a base point $g_0$ as a reference metric
we consider $h\in T_{g_0}\Metmu$ 
as a $g_0$-self-adjoint, tracefree endomorphism $H$ of $TM$. The idea behind the proof is to compute the scalar curvature of 
the Riemannian metrics $\gamma_h(t) \in \Metmu$ by using a locally-defined $g_0$-orthonormal frame $\{e_i\}_{i=1}^{n}$ in which the corresponding
endomorphism $H$ has a block-diagonal form, according to the clustering of the eigenvalues of $H$. Like in the homogeneous case considered in \cite{BWZ}, where the space of invariant metrics is finite-dimensional, 
it is possible to identify negative terms which grow exponentially quickly as $t$ tends to $\infty$. These terms are locally given as a product
of an exponential term and a term of type $g_0([e_1,e_j],e_k)^2$, where the indices $j,k$ correspond to different blocks (of the block diagonalisation of $H$); it therefore becomes useful to ensure that the multiplicities of the eigenvalues of $H$ are as small as possible, so there are many different blocks. 

The open and dense condition in the statement of Theorem \ref{thm0} allows us to restrict attention to choices of $H$ for which certain eigenvalue multiplicities do not occur. 
Ideally, we would like to restrict attention to those $H$ for which all eigenvalues have multiplicity one everywhere, but there are no such $H$ unless the tangent bundle $TM$ splits as a sum of line bundles. As a consequence, 
we must in general accommodate the existence of points where the eigenvalues of $H$ have non-trivial multiplicities.
Therefore, it appears that it is more appropriate to focus on choices of $H$ for which the occurrence of non-trivial multiplicities of the eigenvalues is rare, and where non-trivial multiplicities \textit{do} occur, they do so in a non-degenerate way. We refer to these choices of $H$ as `generic'; see Definition \ref{def:globalgen}. 

Our second main result is:

\begin{mainth}\label{thm1}
 Let $(M,g_0)$ be a compact Riemannian manifold of dimension $n \geq 2$
with volume density $\mu$.
Then, the set of generic directions in $T_{g_0}\Metmu$ is open and dense in the $C^\infty$ topology.
\end{mainth}

 For a given $h$, we use $S_h\subseteq M$ to refer to the multiplicity locus of points where there are not $n=\dim(M)$ distinct eigenvalues of $H$. We also describe the regularity 
 of this set:
 
\begin{mainth}\label{thm2} 
 Let $(M,g_0)$ be a compact Riemannian manifold of dimension \mbox{$n \geq 2$}
with volume density $\mu$.
Then, the multiplicity locus $S_h$ of a generic direction $h \in T_{g_0}\Met_{\mu}$ is either empty or a 
compact Whitney stratified subspace of $M$ of codimension two. 
\end{mainth}

As a consequence of Theorem \ref{thm2}, we find that if $n=2$, the multiplicity locus of a generic direction will be a finite (possibly empty) set of points. 
If $n=3$, the multiplicity locus will be the disjoint union of a finite number of copies of $\mathbb{S}^1$. However, for $n\ge 4$, the multiplicity locus will in general not be a smooth submanifold.
 The strata of the singularity locus are just the points where the ordered eigenvalues of 
$H$ have a fixed multiplicity. The regular part $R_h:=M\backslash S_h$ is an open set of full measure and the tangent bundle
of $R_h$ splits as a sum of line bundles.

We conclude the introduction by discussing connections to topology. 
\begin{question}\label{q2}
To what extent does the topology of the singular set $S_h \subset M$ determine the topology of $M$
for generic $h$?
\end{question}

Question \ref{q2} is motivated by the 
Poincar\'e-Hopf theorem: the sum of the indices
of a smooth vector field with finitely
many zeros equals the Euler characteristic
of $M$ (see \cite[§6]{MW}). 
Notice, that
if $X$ is a generic vector field, 
that is if all the zeros
of $X$ are non-degenerate, then these zeros are 
isolated and have index $\pm 1$. Thus
the singular set of a non-degenerate vector field
provides topological information about 
the underlying smooth manifold.
In particular, as is well-known, there exists
a vector field without zeros if and only if the Euler characteristic vanishes.

Continuing in this direction, we would now like to ask when there exist generic directions $h$ which avoid certain strata.
For instance, if $S_h =\emptyset$, then the tangent bundle $TM$ splits as a sum of line bundles,
i.e., $M$ is line-element parallelizable \cite{MS}. For the converse, see Lemma \ref{lem:lepar}.
Thus, if $M$ is not line-element
parallelizable, then for every generic $h$
we must have $S_h \neq \emptyset$.
If, in addition, $M$
is simply connected, then $M$ is parallelizable, thus orientable, and 
all characteristic classes vanish.
Examples of compact, simply-connected parallelizable manifolds
are compact, simply connected Lie groups
and $S^7$. Compact
examples which are not simply connected
are quotients of Lie groups by a discrete
cocompact lattice. Moreover,
any orientable compact $3$-manifold $M$
is parallelizable \cite{Sti}.
A compact, oriented $4$-manifold is parallelizable
if and only if $w_2(M)=p_1(M)=e(M)=0$ \cite{HH}. For dimensions 
$5,6,7$ see \cite{Thom}.

 If $M$ is not line-element parallelizable, then the next case of interest is when $S_h$ is a union of
compact codimension two submanifolds
where precisely two of the eigenvalues of $H$
agree, i.e., we have only eigenvalue multiplicities
of type $(2,1,\cdots,1),...,(1,\cdots,1,2)$. We assume again for simplicity that
$M$ is simply connected and that only one of the eigenvalue type multiplicities
occurs, say $(1,\cdots,1,2)$. We completely understand when such a multiplicity cannot be avoided: 
$M$ cannot be stably parallelizable
(that is $TM \oplus (M\times \RR)$ is a trivial vector bundle over $M$), because in this scenario, $M$ is parallelizable
if and only if $M$ is stably parallelizable. Indeed, the main theorem of \cite{BK} implies that  if $M^n$ is simply-connected and stably parallelizable, then the span $\sigma(M^n)$ (the maximal number 
of linearly independent vector fields) is either equal to $\sigma(\mathbb{S}^n)$ or equal to $n$. Note that 
$\sigma(\mathbb{S}^n)$ has been computed by Adams \cite{Ad}: if $n\ge 3$, then $\Sigma(\mathbb{S}^n)<n-2$ or $\Sigma(\mathbb{S}^n)$. Our assumption that $(1,1,\cdots,2)$ and $(1,\cdots,1)$ are the only occurring multiplicity types implies that, $\sigma(M^n)\ge n-2$, so in fact $\sigma(M^n)=n$. 

\subsection*{Acknowledgements}
We would like to thank Armin Rainer for sharing with us the proof of Lemma \ref{lem:goodframe}
and Johannes Ebert for very helpful discussions. 
The first-named author was funded by the Deutsche Forschungsgemeinschaft (DFG, German
Research Foundation) under Germany’s Excellence Strategy EXC 2044-390685587, Mathematics
M\"unster: Dynamics-Geometry-Structure, and the Collaborative Research Centre CRC 1442, Geometry: Deformations and Rigidity. The second-named author was supported by the Australian Research Council through grant DE220100919.





\section{Asymptotic behaviour of scalar curvature}\label{preliminaries}

\subsection{Spaces of Riemannian metrics}
\label{sec:spac-riem}
Let $M$ be a smooth, connected, compact manifold of dimension $n\geq 2$.
We denote by $S^2(T^* M)$ the vector bundle of symmetric
$(0,2)$-tensor fields on $M$, and by $\Met \subset \Gamma(S^2 (T^* M))$
the set of all smooth Riemannian metrics on $M$, i.e. $\Met$
is the set of all
smooth, positive definite sections of $S^2(T^* M)$. Since $\Met$ is
an open subset of   $S^2(T^* M)$ in the $C^{\infty}$ topology, it is a smooth Fréchet manifold  (see \cite{Bes}, \mbox{Chapter 4}).
As a consequence,  for any $g \in \Met$, the tangent space $T_g \Met$ 
can be canonically identified with $\Gamma(S^2 (T^* M))$.

We can define a Riemannian metric on $\Met$, called the \emph{Ebin}
or \emph{$L^2$-metric}, by
\begin{equation*}
  \langle h, k \rangle_g = \integral{M}{}(h, k)_g\,{\mu_g},
\end{equation*}
where $h, k \in T_g \Met$ and in local coordinates 
$$
  (h, k)_g = g^{ij} h_{il} g^{lm} k _{jm}\,.
$$ 
Let $\mu:M\to (0,\infty)$ be now any smooth positive density on $M$
with $\int_M 1 \mu=1$.  
We denote by $\Metmu$
the set of all smooth Riemannian metrics 
having volume density $\mu$, that is
\begin{equation*}
  \Metmu = \{ g \in \Met \mid \mu_g = \mu \}.
\end{equation*}
By a result of Moser \cite{Msr}, all Riemannian metrics on $M$
of volume $1$ on $M$ can
be brought into $\Metmu$ via the action of the diffeomorphism group.
The space $\Metmu$ is a smooth Fréchet submanifold of $\Met$ and
for the tangent space of $\Metmu$ at a Riemannian metric $g_0$ we have
\begin{equation*}
  T_{g_0} \Metmu = \{ h \in \Gamma(S^2 (T^* M)) \mid 
  \tr_{g_0}(h) \equiv 0 \}\,,
\end{equation*}
where $\tr_{g_0}(h) = \tr(g_0^{-1} h) = (g_0,h)_{g_0}$. The 
restriction of the $L^2$-metric to $\Metmu$ defines a Riemannian
metric on $\Metmu$, which we still denote by $L^2$.
 Endowed with
this metric, the space $\Metmu$ can be considered 
an infinite-dimensional symmetric space modeled on sections of a fiber bundle over $M$ with
fibers isomorphic to ${\rm SL}(n) / {\rm SO}(n)$.  

Next, for any $h \in T_{g_0}\Metmu$, we have 
\begin{align}\label{eqn:H}
      h(X,Y)=g_0(H\cdot X,Y)
\end{align}
for a unique $g_0$-self-adjoint and traceless endomorphism $H$ of $TM$, 
for all smooth vector fields $X,Y$ on $M$. 
It is well-known, (see \cite{Cl} Section 4.2), 
that the Ebin geodesic $\gamma_h(t)$ of $(\Metmu,L^2)$
with $\gamma_h(0)=g_0$ and $\gamma_h'(0)=h$ is given by
$$
 \gamma_h(t)(\,\cdot \,,\cdot) =g_0(\exp(tH)\,\cdot\,,\cdot).
 $$ 
Here $\exp(H)=\sum_{k=0}^\infty \tfrac{H^k}{k!}$ 
denotes the exponential map of endomorphisms.

\begin{rem}
Having fixed $g_0$ once and for all,
we can consider $H$ a smooth map
\begin{eqnarray}\label{Sym0TM}
    H:M \to \Sym_0(TM)\,,
\end{eqnarray}
where
$$
\Sym_0(TM):=\bigcup_{p \in M}\{F_p \in {\rm End}(T_pM):
    \tr F_p=0 \textrm{ and } F_p \textrm{ is }
    (g_0)_p\textrm{-self-adjoint}\}
$$
is a smooth vector bundle over $M$.
\end{rem}

\subsection{The scalar curvature of Ebin geodesics}
\label{sec:scal-Ebin}

Let $(M,g_0)$ be a Riemannian manifold with volume form $\mu:=\mu_{g_0}$.
In this subsection, we will compute the scalar curvature of the Ebin geodesics $\gamma_h(t)\in \Metmu$ 
in certain open neighbourhoods of $M$ for all $t\ge 0$.

For endomorphisms 
$  S \in {\rm End}(V)  $
of Euclidean vector spaces 
$(V,\langle \,\cdot , \cdot \rangle)$ 
we will use the operator norm
$$
 \Vert S \Vert:=\max \{\Vert S \cdot v\Vert:
 \Vert v \Vert =1\,,\,\,v \in V\}.
 $$
Recall that $\Vert S_1+S_2\Vert \leq \Vert S_1 \Vert + \Vert S_2\Vert$,
$\Vert S_1\cdot S_2\Vert \leq \Vert S_1\Vert
\cdot \Vert S_2\Vert$
and that 
$\Vert A\cdot S\cdot A^t\Vert =\Vert S\Vert$
for all $A \in \rm{O}(V,\langle \,\cdot , \cdot \rangle) $. 

Let $H\in \Gamma({\rm Sym}_0(TM))$ be associated to  $h \in T_{g_0}\Metmu$ via \eqref{eqn:H} and \eqref{Sym0TM}. 
The following lemma, which we largely owe to Armin Rainer, gives us a good block diagonal form for $H$, according to a certain clustering of the eigenvalues.
The constants $C_2(n),\tilde C_2(n)\geq 1$
will be defined in Lemma \ref{lowerboundgIt}
and Lemma \ref{upperboundgIt}.

\begin{lem}\label{lem:goodframe}
Let $(M,g_0)$ be a Riemannian manifold,
$p \in M$ and $(U,x)$ be a coordinate neighbourhood of
$p$ with $x(U)=B_4(0) \subset \RR^n$, such that
on $U$ there exists an $g_0$-orthonormal frame $e_0$.
Moreover,
let $L\geq 2$ and 
 $(m_1,\cdots,m_L)$ be a tuple of positive integers, summing to $n$, and real numbers $(\lambda_1^*,\cdots,\lambda_L^*)$ appearing in (strictly) ascending order, with $\sum_{i=1}^{L}m_i \lambda_i^*=0$. Let
 $$
 r:=\min\{\left|\lambda^*_{i+1}-\lambda^*_i\right|: i=1,\cdots, L-1\}
 \quad\textrm{ and }\quad
  \epsilon := \tfrac{r}{12000 \cdot (C_2(n)+\tilde C_2(n))}.
$$
Now  let $H$ be 
a smooth $g_0$-self-adjoint, traceless endomorphism
of $TM$. 
 Suppose that
for all $q\in x^{-1}(\overline{B_2(0)})$, there are $m_i$ eigenvalues of $H(q)$ in $(\lambda^*_i-\epsilon,\lambda^*_i+\epsilon)$, for all $i=1,\cdots ,L$.
Then,  there exists
 smooth functions $\lambda_1,\cdots,\lambda_L:x^{-1}(B_2(0))\to \RR$, 
and 
smooth traceless block functions $S_i:x^{-1}(B_2(0))\to \Sym_0(m_i)$, as well as a smooth $g_0$-orthonormal 
frame field $e_h=\bigcup_{i=1}^{L}\{e_{i_a}\}_{a=1}^{m_i}$ on $x^{-1}(B_2(0))$ 
such that on
$x^{-1}(B_2(0))$ the endomorphism $H$ takes the form
\begin{align}\label{blockdiagonalH}
     H\vert_{x^{-1}(B_2(0))} &=
              \left( \begin{array}{ccc} \lambda_1I_{m_1}+S_1 && \\ & \ddots & \\ && \lambda_LI_{m_L}+S_L \end{array}\right).
\end{align}
Moreover, on $x^{-1}(B_2(0))$ we have
\begin{align}\label{eigenvaluer2}
     \min\{\left|\lambda_{i+1}-\lambda_i\right|, i=1,\cdots, L-1\}
     &\ge 
     \tfrac{r}{2},\\
     \max\{ \vert \lambda_i - \lambda_i^*\vert, i=1, \cdots, L\} 
     &\leq \epsilon \nonumber ,\\
   \Vert S\Vert:= \max\{ \Vert S_i \Vert: i=1,\cdots, L\} 
    &\leq 4\epsilon.  \nonumber
\end{align}
\end{lem}

\begin{proof}
We pull back the metric $g_0$ and the given
$g_0$-orthonormal frame to a smooth metric 
$\tilde g$ on $B_4(0)$ and a smooth $\tilde g$-orthonormal frame $\tilde e_0$. Using this basis,
we consider 
$ H : B_4(0) \to \Sym_0(n)$ a smooth map. The eigenvalue clustering hypothesis implies that 
for each $x \in B_2(0)$
 and each $i=1,\cdots , L$
there are $m_i$ eigenvalues of $H$ in
$(\lambda^*_i-\epsilon,\lambda^*_i+\epsilon)$.

The characteristic
polynomial $P^x(Z)$ of $H(x)$, $x \in B_4(0)$,
depends smoothly on $x$. Moreover,
the eigenvalue property implies that 
for all $x \in B_2(0)$ we have
 \begin{align*}
     P^x(Z)=\prod_{i=1}^L P_i^x(Z)
 \end{align*}
 for polynomials $P_i^x(Z)$ 
 of degree $m_i$
 whose zeroes all lie in $(\lambda_i^*-\epsilon,\lambda_i^*+\epsilon)$, $i=1,\cdots,L$.
By \cite{Rai}, Lemma 3.2 (see also \cite{AKLM})
the coefficients of the polynomials $P_i^x(Z)$ are smooth in $x$ for all $i=1,...,L$.

Define smooth maps
$L_i :B_2(0) \to \Sym(n)$ for $i=1,...,L$ with $L_i(x) := P_i^x(H(x))$.
By the theorem of      
Cayley-Hamilton we have   $\prod_{i=1}^L L_i(x)=0$ for all $x \in B_2(0)$.
Let $E_i(x)$ denote the kernel of $L_i(x)$
for $x \in B_2(0)$, $i=1,\cdots ,L$. 
Since the smooth family of endomorphisms $L_1(x)$ have constant rank for all $x \in B_2(0)$,
the kernel $E_1(x)$ depends smoothly on $x$,
and consequently
$E_1$ is a smooth subbundle of the trivial vector bundle $B_2(0) \times \RR^n \to  B_2(0)$. Therefore,
 we can find smooth pairwise orthonormal vectors 
$(f_1,\cdots ,f_{m_1}):B_2(0) \to \RR^n$ 
with respect to $\tilde g$
such that for all $x \in B_2(0)$ we have
$E_1(x)=\langle f_1(x),\cdots ,f_{m_1}(x)\rangle_\RR$.
Since the same is true for $i=2,\cdots ,L$, we conclude that $H$ has 
a block decomposition given by $(H_1,\cdots ,H_L)$.
We set $\lambda_i:=\tfrac{1}{m_i}\cdot \tr H_i$
and $S_i:=H_i-\lambda_i \cdot I_{m_i}$,
$i=1,\cdots ,L$. Then $\lambda_i \in 
(\lambda_i^*-\epsilon,\lambda_i^*+\epsilon)$
and $\Vert S_i\Vert \leq 4 \epsilon$, $i=1,\cdots ,L$.
This shows the claim.
\end{proof}

\begin{rem}
Let us mention that $E_1(x)$ is simply the sum of the eigenspaces
of $H(x)$ corresponding to the eigenvalues in $\lambda_1(x),\cdots ,\lambda_{m_1}(x)$
and that the example $U=\mathbb{S}^2$
and $H= I$ shows that we cannot necessarily find
this nice diagonalizing orthonormal basis
globally. Recall also, that in general the
eigenvalues $\lambda_i(x)$ will depend continuously on $x$ but not differentiably: see Section \ref{sec:Casen2}.
\end{rem}

\begin{lem}\label{lem:smoothframe}
Let $H_0$ be a smooth $g_0$-self-adjoint traceless endomorphism of $TM$. Let $L\geq 2$ and 
 $(m_1,\cdots,m_L)$ be a tuple of positive integers, summing to $n$, and real numbers $(\lambda_1^*,\cdots,\lambda_L^*)$ appearing in (strictly) ascending order, with $\sum_{i=1}^{L}m_i \lambda_i^*=0$, and define $r,\epsilon$ as in Lemma \ref{lem:goodframe}. Let $(U,x)$ be a coordinate neighbourhood of
$p$ with $x(U)=B_4(0) \subset \RR^n$ and $x(p)=0$, as in Lemma \ref{lem:goodframe}. In addition, suppose that there are $m_i$ eigenvalues of $H_0(q)$ within $\frac{\epsilon}{10}$ of $\lambda_i^*$ for each $q\in U$. Then there is a compact neighbourhood $K_p$ of $p$ and an open neighbourhood of $H_0$ in the $C^{\infty}$ topology on $\Sym_0(TM)$ on which the $g_0$-orthonormal frame field $e_h$ from Lemma \ref{lem:goodframe}, restricted to $K_p$, can be chosen to depend smoothly on $h$. 
\end{lem}

\begin{proof}
    Define $\tilde{H}_0\in \Sym_0(n)$ to be the diagonal matrix with the same eigenvalues as $H_0(p)$, listed in ascending order. As in the proof of Lemma \ref{lem:goodframe}, the characteristic polynomial $P^{\tilde H}(Z)$ appears as 
    \begin{align*}
        P^{\tilde H}(Z)=\prod_{i=1}^L P_i^{\tilde H}(Z)
    \end{align*}
    for polynomials $P_i^{\tilde H}$ of degree $m_i$ with coefficients depend smoothly on $\tilde H \in \Sym_0(n)$ for $\tilde H$ close to
    $\tilde H_0$. 
   For each $i\in\{1,\cdots,L\}$, the function $L_i:\Sym_0(n)\to \Sym(n)$ sending a matrix $\tilde H$ to $P_i^{\tilde H}(\tilde H)$, has constant rank close to $\tilde{H}_0$, so its kernel depends smoothly on $\tilde H$. Therefore, there is a map $B_i:\Sym_0(n)\to (\mathbb{R}^n)^{m_i}$ which gives us a smoothly-varying basis of the kernel of $L_i(\tilde H)$ for $\tilde H$ close to $\tilde{H}_0$.

Now, for each $H$ close to $H_0$,  we use a $\tilde g_0$-orthonormal
frame $\tilde e_0$ on $B_2(0)$ to consider $H$ as a family
of symmetric endomorphisms. For those choices of $H$ for which there are $m_i$ eigenvalues of $H(q)$ within $\frac{\epsilon}{5}$ of $\lambda_i^*$ for each $q\in U$, we can apply 
Lemma \ref{lem:goodframe} to construct
a $g_0$-orthonormal frame $e_h$ on $U=x^{-1}(B_2(0))$
in which $H$ takes the block-diagonal form \eqref{blockdiagonalH}. In fact, since the above $B_i$ maps are smooth, for each $h\in \Sym_0(TM)$ close to $h_0$, there is a basis $f_h$, block-diagonalising $H$ as in Lemma \ref{lem:goodframe}, which depends smoothly on $h$. Applying Gram-Schmidt implies that our $g_0$-orthonormal frame $e_h$ can be chosen to vary smoothly with $H$ close to $H_0$. This implies the result. 
\end{proof}

\begin{rem}
For our applications it is actually sufficent 
that $e_h$ can be chosen to depend smoothly on $h$ in the 
$C^1$ topology.
\end{rem}

To compute the scalar curvature of $g_t:=\gamma_h(t)$ on a neighbourhood $K_p$ from Lemma \ref{lem:smoothframe}, 
it is convenient to find an orthonormal frame for $g_t$; the most natural choice is given by $\bigcup_{i=1}^{L}\{e^t_{i_a}\}_{a=1}^{m_i}$, where
    \begin{align*}
    e_{i_{a}}^t
    =
    \tfrac{1}{\sqrt{\alpha_i(t)}} \cdot \sum_{\tilde a=1}^{m_i}
    \exp\big(-\tfrac{S_i t}{2}\big)_{a\tilde a}\cdot e_{i_{\tilde a}},  \ \text{for} \  i=1,\cdots ,L \ \text{and} \ a=1,\cdots ,m_i,
\end{align*}
with
\begin{align*}
    \alpha_i(t)=\text{exp}(\lambda_i t).
    \end{align*}
Then we define the Christoffel symbols 
$$
(\Gamma_t)_{i_a,j_b}^{k_c}
:=
g_t(\nabla^t_{e^t_{i_a}}e^t_{j_b},e^t_{k_c}),
$$
where $\nabla^t$ is the Levi-Civita connection of $g_t$.

\begin{lem}
The Riemann curvature tensor $\Rc^t$ of $g_t$ on $K_p$ is given by \begin{eqnarray*}
  \lefteqn{g_t(\Rc^t_{e^t_{i_a}e^t_{j_b}}e^t_{k_c},e^t_{l_d})}&&\\
  &= &
  e^t_{i_a}(\Gamma_t)_{j_b,k_c}^{l_d} +e^t_{j_b}(\Gamma_t)_{i_a,l_d}^{k_c} 
   +\sum_{q=1}^{L}\sum_{p=1}^{m_q} \big[
    (\Gamma_t)_{j_b,k_c}^{q_p} (\Gamma_t)_{i_a,q_p}^{l_d}
  - (\Gamma_t)_{i_a,k_c}^{q_p} (\Gamma_t)_{j_b,q_p}^{l_d}\\
  &&
  -\big((\Gamma_t)_{i_a,j_b}^{q_p}-(\Gamma_t)_{j_b,i_a}^{q_p}\big) (\Gamma_t)_{q_p,k_c}^{l_d} \big]\,.
\end{eqnarray*}
\end{lem}

\begin{proof}
We have \begin{align*}
  g_t(\Rc_{e^t_{i_a}e^t_{j_b}}e^t_{k_c},e^t_{l_d})=
  g_t(\nabla^t_{e^t_{i_a}}\nabla^t_{e^t_{j_b}}e^t_{k_c}-\nabla^t_{e^t_{j_b}}\nabla^t_{e^t_{i_a}}e^t_{k_c} -\nabla^t_{[e^t_{i_a},e^t_{j_b}]}e^t_{k_c},e^t_{l_d})\,.
\end{align*}
The result then follows from the observation $[e^t_{i_a},e^t_{j_b}]=\nabla_{e^t_{i_a}}e^t_{j_{b}}-\nabla_{e^t_{j_b}}e^t_{i_a}$ so that

\begin{align*}
    g^t(\nabla^t_{[e^t_{i_a},e^t_{j_b}]}e^t_{k_c},e^t_{l_d})=\sum_{q=1}^{L}\sum_{p=1}^{m_q}((\Gamma_t)_{i_a,j_b}^{q_p}-(\Gamma_t)_{j_b,i_a}^{q_p})(\Gamma_t)_{q_p,k_c}^{l_d}.
\end{align*}
This shows the claim.
\end{proof}

As a consequence, we obtain 

\begin{cor}\label{scalrformula}
The scalar curvature $R(g_t)$ of $g_t$ on $K_p$ is given by 
\begin{align*}
R(g_t)=2\sum_{i_a,j_b}e^t_{i_a}(\Gamma_t)^{i_a}_{j_b,j_b}-\sum_{i_a,j_b,k_c}\left((\Gamma_t)_{i_a,i_a}^{k_c}(\Gamma_t)_{j_b,j_b}^{k_c}+(\Gamma_t)_{i_a,k_c}^{j_b}(\Gamma_t)_{k_c,i_a}^{j_b}\right).
\end{align*}
\end{cor}

\subsection{Asymptotic behavior of the scalar curvature}
\label{subsec:Asbe}

To compute the asymptotic behavior of the scalar curvature of $g_t$
on a neighbourhood $K_p$ from Lemma \ref{lem:smoothframe}, 
we define 
$$
 (g_t)^{i_a,j_b,k_c}:=g_t([e^t_{i_a},e^t_{j_b}],e^t_{k_c})
 $$ 
 for
  $1 \leq i,j,k \leq L$ and
 $1\leq a \leq m_i$, $1 \leq b \leq m_j$, 
 $1\leq c \leq m_k$.
 Recall that by the Koszul formula we have
 \begin{align}\label{ChristoffelSymbols}
     2(\Gamma_t)^{k_c}_{i_a,j_b}=(g_t)^{i_a,j_b,k_c}+(g_t)^{k_c,i_a,j_b}+(g_t)^{k_c,j_b,i_a}.
 \end{align}
  Therefore 
\begin{eqnarray*}
 \lefteqn{ (g_t)^{i_a,j_b,k_c}=}&&\\
&=&
g_0([e^t_{i_a},e^t_{j_b}],\text{exp}(t\lambda_k I_{m_k}+t S_k)\cdot e^t_{k_c})\\
 &=&
    \sum_{\tilde a,\tilde b,\tilde c}g_0
    \Big(
    \Big[ \exp\big( \tfrac{tS_i}{-2}\big)_{a\tilde a} \cdot  \tfrac{e_{i_{\tilde a}}}{\sqrt{\alpha_i(t)}}  ,
    \exp\big(\tfrac{tS_j}{-2}\big)_{b\tilde b} \cdot  \tfrac{e_{j_{\tilde b}}}{\sqrt{\alpha_j(t)}}\Big],\exp\big( \tfrac{tS_k}{2}\big)_{c\tilde c}\cdot \tfrac{e_{k_{\tilde c}}}{\sqrt{\alpha_k(t)}} 
    \Big).
\end{eqnarray*}
In order to understand the behaviour of the scalar curvature 
$R(g_t)$ as $t$ becomes large, 
we find it convenient to break the $(g_t)^{i_a,j_b,k_c}$ terms up using the rule $[f_1X_1,f_2X_2]=f_1(X_1f_2)X_2-f_2(X_2f_1)X_1+f_1f_2[X_1,X_2]$:
\begin{align*}
  (g_t)^{i_a,j_b,k_c}&=  (g_t)^{i_a,j_b,k_c}_{I}+(g_t)^{i_a,j_b,k_c}_{II}+(g_t)^{i_a,j_b,k_c}_{III},
\end{align*}
where
\begin{align*}
    (g_t)^{i_a,j_b,k_c}_{I}
    &=
    \sqrt{\tfrac{\alpha_k(t)}{\alpha_i(t)\alpha_j(t)}}\cdot
    \sum_{\tilde a,\tilde b,\tilde c}\exp
    \big(\tfrac{tS_i}{-2}\big)_{a\tilde a}\cdot
    \exp\big(\tfrac{tS_j}{-2}\big)_{b\tilde b}\cdot
    \exp \big(\tfrac{tS_k}{2}\big)_{c\tilde c}\cdot (g_0)^{i_{\tilde a},j_{\tilde b},k_{\tilde c}},\\
   (g_t)^{i_a,j_b,k_c}_{II} 
    &=
    \tfrac{\delta_{jk}\sqrt{\alpha_k(t)}}{\sqrt{\alpha_i(t)}}
    \sum_{\tilde a,\tilde b}
    \exp\big(\tfrac{tS_i}{-2}\big)_{a\tilde a}\cdot 
   e_{i_{\tilde a}}\big(\exp\big(\tfrac{t\lambda_j }{-2}I_{m_j}+\tfrac{tS_j}{-2}\big)\big)_{b\tilde b}\cdot 
    \exp\big(\tfrac{tS_j}{2}\big)_{c\tilde b}, \\
    (g_t)^{i_a,j_b,k_c}_{III}
    &=
    \tfrac{-\delta_{ik}\sqrt{\alpha_k(t)}}{\sqrt{\alpha_j(t)}}
    \sum_{\tilde a,\tilde b}
    e_{j_{\tilde b}}\big(\exp\big(\tfrac{t\lambda_i }{-2}I_{m_i}+\tfrac{tS_i}{-2}\big)\big)_{a\tilde a}\cdot
      \exp\big(\tfrac{tS_j}{-2}\big)_{b\tilde b}\cdot
    \exp\big(\tfrac{tS_i}{2}\big)_{c\tilde a}.
\end{align*}
We observe the following result about the relationship between the $(g_t)_{I}^{i_a,j_b,k_c}$ terms and the $(g_0)^{i_a,j_b,k_c}$ terms.

\begin{lem}\label{lowerboundgIt}
Let $i,j,k\in \{1,\cdots,L\}$. Then, there exists constants $C_1(n)>0$ and $C_2(n)\geq 1$ so that for all $t \geq 0$
\begin{align*}
    \sum_{a,b,c}\left((g_t)_I^{i_a,j_b,k_c}\right)^2
    \ge 
    C_1(n)\cdot \tfrac{ \alpha_k(t)}{\alpha_i(t)\alpha_j(t)} \cdot 
    e^{-3C_2(n) \cdot \left\Vert S\right\Vert \cdot  t}
    \cdot \sum_{a,b,c}(g_0^{i_a,j_b,k_c})^2. 
\end{align*}
\end{lem}

\begin{proof}
We have 
\begin{align*}
    (g_0)^{i_a,j_b,k_c}
    &=
\sqrt{\tfrac{\alpha_i(t)\alpha_j(t)}{\alpha_k(t)}}\cdot
    \sum_{\tilde a,\tilde b,\tilde c}\exp
    \big(\tfrac{tS_i}{2}\big)_{a\tilde a}\cdot
    \exp\big(\tfrac{tS_j}{2}\big)_{b\tilde b}\cdot
    \exp \big(\tfrac{tS_k}{-2}\big)_{c\tilde c}\cdot (g_t)_I^{i_{\tilde a},j_{\tilde b},k_{\tilde c}}.
\end{align*}
We estimate 
\begin{align*}
    \left|(g_0)^{i_a,j_b,k_c}\right|\le C_3(n) \cdot \sqrt{\tfrac{\alpha_i(t)\alpha_j(t)}{\alpha_k(t)}}
    \cdot e^{3C_2(n)\cdot \left\Vert S\right\Vert
    \cdot t}\cdot 
    \max_{\tilde{a},\tilde{b},\tilde{c}}\left|(g_t)_I^{i_{\tilde{a}},j_{\tilde{b}},k_{\tilde{c}}}\right|.
\end{align*}
Squaring and summing over $a,b,c$ gives the result. 
\end{proof}

The above lemma implies that the terms of type $I$ can experience strong exponential growth, especially for the terms with $i\le j<k$ and $\left|\left|S\right|\right|$ small. 

The following lemma gives upper bounds for terms of type $I$, $II$, and $III$. Since terms of type $II$ and $III$ vanish unless certain indices are identical, we conclude that terms of type $II$ and $III$ can \textit{never} experience exponential growth as strong as the type $I$ terms can. 

\begin{lem}\label{upperboundgIt}
Let $1 \leq i,j,k \leq L$, $1 \leq a \leq m_i$, $1\leq b\leq m_j$ and $1 \leq c \leq m_k$.
Then there exists constants $\tilde C_1(n)>0$
and $\tilde C_2(n) \geq 1$ such that for all $t\geq 0$ we have
\begin{align*}
    &\big\vert
    (g_t)_I^{i_a,j_b,k_c}
    \big\vert+ \big\vert
    (g_t)_{II}^{i_a,j_b,k_c}
    \big\vert
     +\big\vert
    (g_t)_{III}^{i_a,j_b,k_c}
    \big\vert\\
    &\leq 
    \tilde C_1(n)\cdot (1+t) \cdot \sqrt{\tfrac{\alpha_k(t)}{\alpha_i(t)\alpha_j(t)}}\cdot 
    e^{3\tilde C_2(n)\cdot \Vert S\Vert \cdot t}\cdot
    \big(
   \max_{a,b,c}\left|(g_0)^{i_a,j_b,k_c}\right|
   + C(h) \Big),
\end{align*}
where 
$$
C(h)= C(
\Vert \nabla_{g_0} \lambda_l \Vert ,\Vert \nabla_{g_0} S_l\Vert: l=1,\cdots, L)
$$
depends continuously on $h$
in the $C^1$-topology.
\end{lem}

\begin{proof}
The result is obvious for the `algebraic' curvature term $(g_t)_I^{i_a,j_b,k_c}$. For the other two, note that for a vector field $v$ on $K_p$ 
we compute
\begin{eqnarray*}
    \lefteqn{v\big(\tfrac{1}{\sqrt{\alpha_j(t)}}\cdot   \exp\big(\tfrac{tS_j}{-2}\big)_{b\tilde b}\big)=}&&\\
    &=&
     v \exp(-\tfrac{t\lambda_j}{2})\cdot   \exp\big(\tfrac{tS_j}{-2}\big)_{b\tilde b}
     + \exp(-\tfrac{t\lambda_j}{2})\cdot   v \exp\big(\tfrac{tS_j}{-2}\big)_{b\tilde b} \\
      &=&
      \tfrac{1}{\sqrt{\alpha_j(t)}}\cdot
      \Big(
     -\tfrac{t}{2}\cdot v\lambda_j \cdot    \exp\big(\tfrac{tS_j}{-2}\big)_{b\tilde b}
     +
      \big(\sum_{k=0}^\infty \tfrac{t^k}{-2^k k!}\cdot v (S_j^k) \big)_{b\tilde b} \Big).
\end{eqnarray*}
Next we use $\Vert S_1 \cdot S_2 \Vert \leq \Vert S_1\Vert \cdot \Vert S_2\Vert $ for the operator norm of symmetric matrices.
This yields, using the triangle inequality, the estimate
\begin{align*}
     \Vert v \exp(tS_j) \Vert 
     &=
        \big\Vert v \sum_{k=1}^\infty \tfrac{(tS_j)^k}{k!} \big\Vert\\
        &=\Big\Vert
        \sum_{k=1}^{\infty}\tfrac{t^k}{k!}\left(v(S_j) S^{k-1}_j+S_j v(S_j)S^{k-2}_j+\cdots+S^{k-1}_jv(S_j)\right)\Big\Vert \\
    &\leq
      t \cdot \Vert v (S_j) \Vert \cdot e^{t\Vert S_j\Vert},
\end{align*}
which shows the claim.
\end{proof}

Before we turn to the main theorem of this section, recall
that for $h \in T_{g_0}\Metmu$ we defined the $g_0$-self-adjoint
endomorphism $H$ of $TM$ by $h(X,Y)=g_0(H\cdot X,Y)$
for all vector fields $X,Y$ on $M$.


\begin{thm}\label{curvatureasy}
Let $(M,g_0)$ be a compact Riemannian manifold, let $p \in M$ and let
$h \in T_{g_0}\Metmu$. Suppose that
$h$ satisfies the assumptions of 
 Lemma \ref{lem:goodframe}.
If there exists $\delta_p>0$
such that for all $q \in x^{-1}(\overline{B_1(0)})$
\begin{eqnarray}\label{eqn:deltap}
\sum_{1\le j<k}\sum_{a\le m_1,b\le m_j,c\le m_k}(g_0^{1_a,j_b,k_c})^2\ge \delta_p,   
\end{eqnarray}
then there exist constants 
$ C_1(h),C_2(h),T(h)>0$ such that for each
$t>T(h)$ we have 
\begin{align*}
    R(\gamma_h(t))\le -C_1(h)\cdot  e^{C_2(h) t}
\end{align*}
on $x^{-1}(\overline{B_1(0)})$.
Moreover, the constants $C_1,C_2,T$ can be chosen to depend continuously on $h$ in $C^2$-topology and continuously on $\delta_p$.
\end{thm}

\begin{proof}
By the hypothesis of Lemma \ref{lem:goodframe}, there are $m_i$ eigenvalues of $H(q)$ within $\epsilon$ of $\lambda_i^*$ for each $i=1,\cdots,L$ and $q\in x^{-1}(\overline{B_1(0)})$, where $\epsilon=\frac{r}{12000\cdot (C_2(n)+\tilde{C}_2(n))}$.
The definition of $r$ implies that $\lambda_i\ge \lambda_1+\frac{r}{2}$ for all $i\neq 1$; combining this with $\sum_{i=1}^{L}m_i\lambda_i=0$ gives  
\begin{align*}
    \lambda_1\le -\tfrac{r}{2n}(n-m_1)\leq -\tfrac{r}{2n}<0.
\end{align*}
Then by the very definition of $\epsilon$, we have 
 $$
  3(C_2(n)+\tilde C_2(n))\cdot \Vert S\Vert 
 \leq   
12 (C_2(n)+\tilde C_2(n))\cdot  \epsilon
  \leq \tfrac{r}{1000}
$$
for all $1 \leq i,j,k \leq L$ and all $q \in x^{-1}(\overline{B_1(0)})$.
We will use the formula for scalar curvature, Corollary \ref{scalrformula}, and estimate the terms using
Lemma \ref{lowerboundgIt} and Lemma \ref{upperboundgIt}. We will focus on terms
which have an exponential growth of order at least
\begin{align}\label{eqn:exporder}
     \exp(t\cdot (-\lambda_1 +\tfrac{r}{2})). 
\end{align}
Firstly, it is clear that the first terms of $R(g_t)$
in Corollary \ref{scalrformula} have an exponential growth
of order less than \eqref{eqn:exporder}. To this end note, that $(\Gamma_t)^{i_a}_{j_b,j_b}=(g_t)^{i_a,j_b,j_b}$. It follows
as in the proof 
of Lemma \ref{upperboundgIt}
that both  $(\Gamma_t)^{i_a}_{j_b,j_b}$ and $e_{i_a}^{t}(\Gamma_t)^{i_a}_{j_b,j_b}$ have lower order exponential growth.
When obtaining upper bounds for such terms, a constant $C(h)$
will appear which depends continuously on $h$ in $C^2$-topology.

Secondly, the middle terms of $R(g_t)$ in
Corollary \ref{scalrformula} are
\begin{align*}
(\Gamma_t)_{i_a,i_a}^{k_c}\cdot
(\Gamma_t)_{j_b,j_b}^{k_c}
=
(g_t)^{k_c,i_a,i_a}(g_t)^{k_c,j_b,j_b}.
\end{align*}
By the very definition of $(g_t)^{k_c,i_a,i_a}$
these terms are of lower exponential order than \eqref{eqn:exporder}.
When obtaining upper bounds for such terms, a constant $C(h)$
will appear which depends continuously on $h$ in $C^1$-topology.

Finally, we examine the last terms of   $R(g_t)$
in Lemma \ref{scalrformula}. 
Using \eqref{ChristoffelSymbols}, we compute 
\begin{eqnarray*}
\lefteqn{4(\Gamma_t)_{i_a,j_b}^{k_c}\cdot  (\Gamma_t)_{j_b,i_a}^{k_c}=}&&\\
&=&
   \left((g_t)^{i_a,j_b,k_c}
   \!+\!
   (g_t)^{k_c,i_a,j_b}
   \!+\!(g_t)^{k_c,j_b,i_a}\right)
   \left((g_t)^{j_b,i_a,k_c}
   \!+\!
   (g_t)^{k_c,j_b,i_a}
   \!+\!(g_t)^{k_c,i_a,j_b}\right)\\
    &=&
    \left((g_t)^{i_a,j_b,k_c}
    \!+\!
    (g_t)^{k_c,i_a,j_b}
    \!+\!(g_t)^{k_c,j_b,i_a}\right)
    \left(-(g_t)^{i_a,j_b,k_c}
    \!+\!
    (g_t)^{k_c,j_b,i_a}
    \!+\!
    (g_t)^{k_c,i_a,j_b}\right)\\
    &=&
    -((g_t)^{i_a,j_b,k_c})^2+((g_t)^{k_c,j_b,i_a})^2+((g_t)^{k_c,i_a,j_b})^2+2(g_t)^{k_c,i_a,j_b}(g_t)^{k_c,j_b,i_a}.
\end{eqnarray*}
First observe that the last term is not of the exponential order of \eqref{eqn:exporder}. Indeed, Lemma \ref{upperboundgIt} implies that 
\begin{align*}
   &\left|(g_t)^{k_c,i_a,j_b}(g_t)^{k_c,j_b,i_a}\right|^{\frac{1}{2}}
   \le \tfrac{\tilde C_1(n)(1+t)}{\alpha_k} 
   \cdot 
    e^{3\tilde C_2(n)\cdot \Vert S\Vert \cdot t}\cdot
    \big(\max_{a,b,c}\left|(g_0)^{i_a,j_b,k_c}\right|+ C(h)\big),
\end{align*}
where $ C(h)$ depends continuously on $h$ in $C^1$-topology.
Note now, that the smallness of $\Vert S\Vert $ as given in Lemma \ref{lem:goodframe} guarantees the growth is not excessive.  
Next, symmetrization implies up to terms of lower exponential order (lot) that
\begin{eqnarray*}
   \lefteqn{4\left((\Gamma_t)_{i_a,j_b}^{k_c}
    (\Gamma_t)_{j_b,i_a}^{k_c}+(\Gamma_t)_{i_a,k_c}^{j_b}
    (\Gamma_t)_{k_c,i_a}^{j_b}+(\Gamma_t)_{k_c,j_b}^{i_a}
    (\Gamma_t)_{j_b,k_c}^{i_a}\right)=}&&\\
    &=&
    ((g_t)^{i_a,j_b,k_c})^2+((g_t)^{k_c,j_b,i_a})^2+((g_t)^{k_c,i_a,j_b})^2 + lot.
\end{eqnarray*}
The key observation is that the 
square of terms of type I on $x^{-1}(\overline{B_1(0)})$
have the property 
that by Lemma \ref{lowerboundgIt}
their exponential growth is as least as large as
in \eqref{eqn:exporder} for $i=1$ and $j<k$,
using \eqref{eqn:deltap}.
By Lemma \ref{upperboundgIt}, the terms containing
at least one type II or III term are of lower
exponential growth. This shows the claim.
\end{proof}

We find it appropriate to now define the set of choices of $H$ for which Theorem \ref{curvatureasy} is applicable on \textit{all} of $M$. 
\begin{defin}\label{candidateopendense}
Let $(M,g_0)$ be a compact Riemannian manifold. 
We define $\Yod \subset T_{g_0}\Metmu$ 
to be the set of $g_0$-self-adjoint and
traceless endomorphisms $H$ so that
for each $H \in \Yod$ and each $p\in M$, there is an open co-ordinate neighbourhood $U_p$ of $p$ satisfying the following:
\begin{itemize}
    \item There exist $L\geq 2$,
    $\lambda_1^*< \cdots < \lambda_L^*$
    and $m_1,\cdots ,m_L$ as in Lemma \ref{lem:goodframe}, such that
      $H$  satisfies the hypothesis
    of Lemma \ref{lem:goodframe}. In particular we assume uniform eigenvalue clustering on $\overline{U}_p$.
    \item There is a compact neighbourhood
    $K_p \subset U_p$ with $p$ in its interior
    such that 
    \begin{eqnarray}\label{eqn:posstr}
    \sum_{1\leq j <k\leq L}
       \sum_{b\leq m_j,c\leq m_k}
     \big( (g_0)^{1_1,j_b,k_c}\big)^2
     >0  
    \end{eqnarray}
    on all of $K_p$, for some block-diagonalising frame $\bigcup_{i=1}^{L}\{e_{i_a}\}_{a=1}^{m_i}$ as in Lemma \ref{lem:goodframe}. 
\end{itemize}
\end{defin}
\begin{rem}
    It should be noted that \eqref{eqn:posstr} is `tensorial', i.e., does not depend on the specific choice of block-diagonalising frame. 
\end{rem}

This set $\Yod$ is our candidate for the open and dense set in Theorem \ref{thm0} because, by Theorem \ref{curvatureasy}, it consists entirely of choices of $H$ for which $\lim_{t\to \infty}R(g_t)=-\infty$ uniformly on $M$.
We make the initial observation that this set is open. 

\begin{lem}\label{lem:candidateopen}
The set $\Yod \subset T_{g_0} \Metmu$ is open in the Whitney $C^{\infty}$ topology.
\end{lem}
\begin{proof}
Let $H\in \Yod$. For each $p\in M$, find the compact $K_p$ and frame on $\overline{U_p}$ granted
 by the definition of $\Yod$.
 Note that the open eigenvalue clustering
 of $H$ on the compact set $\overline{U_p}$
 is open in $C^\infty$ topology.
 Moreover condition \eqref{eqn:posstr} holding on $K_p$ survives perturbation in $C^\infty$ topology
 by Lemma \ref{lem:smoothframe}.

Since $M$ is compact and $p$ is arbitrary, $M$ can be covered by finitely many of these sets $(K_{p_i})_{i=1}^N$ for some
$N \in \NN$. Note that there exists an open neighbourhood $\mathcal{Y}_H \subset T_{g_0} \Metmu $ of $H$ in $T_{g_0} \Metmu$
in the $C^\infty$ topology
such that all $\tilde H \in \mathcal{Y}_H$
satisfy the very same open $\epsilon$-eigenvalue clustering on $K_{p_i}$ 
for all $i=1,...,N$, as well as the positivity condition \eqref{eqn:posstr} of the
 above structure constants. 
\end{proof}

The main goal of this paper is to establish that, in addition to being open, $\Yod$ is dense in $T_{g_0}\Metmu$ on manifolds of dimension at least five.

\section{The diagonal case}

In this section, we examine the formulas for scalar curvature more carefully in the special case that we can choose $m_i=1$ for all $i$ in  Lemma \ref{lem:goodframe}. In this case there is a $g_0$-orthonormal frame $\{e_i\}_{i=1}^{n}$ which diagonalizes $H$, instead of merely block-diagonalising. We can write $H=\text{diag}(\lambda_1,\cdots,\lambda_n)$, so that a $g_t$-orthonormal frame is given by $e_i^t=\frac{e_i}{\sqrt{\alpha_i}}$, where $\alpha_i=e^{\lambda_i t}$. Similarly to the earlier block diagonal case, we write 
\begin{align*}
    (g_t)^{ijk}=g_t([e_i^t,e_j^t],e_k^t), \qquad 2(\Gamma_t)_{ij}^{k}=(g_t)^{ijk}+(g_t)^{kij}+(g_t)^{kji},
\end{align*}
so the scalar curvature becomes 
\begin{align*}
    R(g_t)=2\sum_{i,j}e^t_{i}(\Gamma_t)^{i}_{jj}-\sum_{ijk}\left((\Gamma_t)_{ii}^{k}(\Gamma_t)_{jj}^{k}+(\Gamma_t)_{ik}^{j}(\Gamma_t)_{ki}^{j}\right).
\end{align*}
To understand this formula, we again write
\begin{align*}
    (g_t)^{ijk}=g_t([e_i^t,e_j^t],e_k^t),=(g_t)_{I}^{ijk}+(g_t)_{II}^{ijk}+(g_t)_{III}^{ijk},
\end{align*}
where
\begin{align*}
    (g_t)_{I}^{ijk}=\sqrt{\tfrac{\alpha_k}{\alpha_i\alpha_j}}(g_0)^{ijk}, \qquad (g_t)_{II}^{ijk}=-t\tfrac{\delta_{jk}e_i(\lambda_j)}{2\sqrt{\alpha_i}}, \qquad (g_t)_{III}^{ijk}=t\tfrac{\delta_{ik}e_j(\lambda_i)}{2\sqrt{\alpha_j}}.
\end{align*}
Then we calculate the scalar curvature terms:
\begin{align*}
    \sum_{i,j}(e_i^t)(\Gamma_t)^i_{jj}&=\sum_{i\neq j}\tfrac{e_i(g_t^{ijj})}{\sqrt{\alpha_i}}\\
    &=\sum_{i\neq j}\tfrac{e_i((g_t)_I^{ijj})+e_i((g_t)_{II}^{ijj})}{\sqrt{\alpha_i}}\\
    &=\sum_{i\neq j}\tfrac{e_i(g_0)^{ijj}-\frac{t}{2}(g_0)^{ijj}e_i\lambda_i-\frac{t}{2}e_ie_i\lambda_j+\frac{t^2}{4}e_i\lambda_j e_i\lambda_i}{\alpha_i}.
    \end{align*}
    Then, using the trace-free condition, we find 
    \begin{align*}
       2\sum_{i,j}(e_i^t)(\Gamma_t)^{i}_{jj}&= \sum_{i\neq j}\tfrac{2e_i(g_0)^{ijj}-t((g_0)^{ijj}e_i\lambda_i+e_ie_i\lambda_j)}{\alpha_i}-\tfrac{t^2}{2}\sum_i \tfrac{(e_i\lambda_i)^2}{\alpha_i}.
    \end{align*}
Next, since $(\Gamma_t)_{ij}^j=0$, we conclude that 
\begin{eqnarray*}
   \lefteqn{ \sum_{i,j,k}\left((\Gamma_t)_{ii}^k(\Gamma_t)_{jj}^{k}+(\Gamma_t)_{ik}^{j}(\Gamma_t)_{ki}^{j}\right)=}&&\\
   &=&
   \sum_{PD}\left((\Gamma_t)_{ii}^k(\Gamma_t)_{jj}^{k}+(\Gamma_t)_{ik}^{j}(\Gamma_t)_{ki}^{j}\right)+2\sum_{i\neq j}(\Gamma_t)_{ii}^{j}(\Gamma_t)_{ii}^{j},
\end{eqnarray*}
where $PD$ denotes the collection of triplets $(i,j,k)$ with $i,j,k$ pairwise distinct. 
We compute 
\begin{align*}
    \sum_{PD}(\Gamma_t)^k_{ii}(\Gamma_t)^k_{jj}&=\sum_{PD}(g_t)^{kii}(g_t)^{kjj}\\
    &=
    \sum_{PD}((g_t)_{I}^{kii}+(g_t)_{II}^{kii})((g_t)_I^{kjj}+(g_t)_{II}^{kjj})\\
    &=
    \sum_{PD}\left(\tfrac{(g_0)^{kii}}{\sqrt{\alpha_k}}-\tfrac{t}{2\sqrt{\alpha_k}}e_k(\lambda_i)\right)
    \cdot
    \left(\tfrac{(g_0)^{kjj}}{\sqrt{\alpha_k}}-\tfrac{t}{2\sqrt{\alpha_k}}e_k(\lambda_j)\right).
\end{align*}
Notice again that the trace constraint implies 
\begin{eqnarray*}
   \lefteqn{ \sum_{PD}(\Gamma_t)_{ii}^{k}(\Gamma_t)_{jj}^{k}=}&&\\
    &=&
    \sum_{PD}\tfrac{1}{\alpha_k}\left((g_0)^{kii}(g_0)^{kjj}
    -\tfrac{t}{2}((g_0)^{kjj}e_k\lambda_i+(g_0)^{kii}e_k\lambda_j)+\tfrac{t^2}{4}e_k(\lambda_i)e_k(\lambda_j)\right)\\
    &=&
    \sum_{PD}\tfrac{1}{\alpha_k}\left((g_0)^{kii}(g_0)^{kjj}-\tfrac{t}{2}((g_0)^{kjj}e_k\lambda_i+(g_0)^{kii}e_k\lambda_j)\right)\\
    &&
    -\sum_{j\neq k}\tfrac{t^2}{4\alpha_k}e_k(\lambda_j)e_k(\lambda_j+\lambda_k)\\
    &=&
    \sum_{PD}\tfrac{1}{\alpha_k}\left((g_0)^{kii}(g_0)^{kjj}-\tfrac{t}{2}((g_0)^{kjj}e_k\lambda_i+(g_0)^{kii}e_k\lambda_j)\right)\\
    &&
    -\sum_{j\neq k}\tfrac{t^2}{4\alpha_k}e_k(\lambda_j)e_k(\lambda_j+\lambda_k).
\end{eqnarray*}
Similarly, 
\begin{align*}
    \sum_{j\neq k}(\Gamma_t)_{jj}^{k}(\Gamma_t)^k_{jj}
    &=
    \sum_{j\neq k}\left(\tfrac{(g_0)^{kjj}}{\sqrt{\alpha_k}}-\tfrac{t}{2\sqrt{\alpha_k}}e_k(\lambda_j)\right)^2.
\end{align*}
Finally, 
\begin{eqnarray*}
   \lefteqn{ \sum_{PD}(\Gamma_t)^{j}_{ik}(\Gamma_t)^{j}_{ki}=}&&\\
   &=&\tfrac{1}{4}
   \sum_{PD} ((g_t)_I^{ikj}+(g_t)_I^{jik}+(g_t)_I^{jki})(-(g_t)_I^{ikj}+(g_t)_I^{jik}+(g_t)_I^{jki})\\
    &=&\tfrac{1}{4}
    \sum_{PD} \left(\tfrac{\alpha_j}{\alpha_i\alpha_k}(g_0^{kij})^2+\tfrac{\alpha_i}{\alpha_j\alpha_k}(g_0^{kji})^2-\tfrac{\alpha_k}{\alpha_i\alpha_j}(g_0^{ijk})^2+\tfrac{2}{\alpha_k}g_0^{kij}g_0^{kji}\right).
\end{eqnarray*}

\begin{cor}\label{cor:scalfor}
We have
\begin{align*}
    R(g_t)=\sum_{PD}e^{(\lambda_k-\lambda_i-\lambda_j)t}\cdot F_{ij}^{k}+\sum_{i=1}^{n}e^{-\lambda_i t}\cdot (t^2F_i^{(2)}+tF_i^{(1)}+F_i^{(0)}),
\end{align*} where 
\begin{align*}
    F_{ij}^{k}
    &=
    -\tfrac{1}{4}g_0([e_i,e_j],e_k)^2, \\ F_i^{(2)}
    &=
    -\tfrac{3}{4}(e_i\lambda_i)^2
    -\tfrac{1}{4}\sum_{j\neq i}(e_i\lambda_j)^2,\\
      F_i^{(1)}&=                  e_i(e_i \lambda_i)+ \sum_{\substack{j=1 \\ j \neq i}}^n  \gijj  
        ( e_i \lambda_j - 2 e_i \lambda_i), \\
    F_i^{(0)}&= 
    2 \sum_{\substack{j=1 \\ j \neq i}}^n \big( e_i \gijj  -  (\gijj)^2  \big)        
       + \sum_{\substack{j,k=1 \\ i \neq j \neq k}}^n 
     \big(  \tfrac{1}{2} \cdot \gijk \cdot \gkij   - \gijj \cdot \gikk\big).
\end{align*}
\end{cor} 
In fact, this formula applies whenever we can find a smooth $g_0$-orthonormal frame of eigenvectors of $H$; we do not need to assume that these eigenvalues are everywhere pairwise distinct. Such examples can even be found globally under the assumption that $M$ is line-element parallelizable.

\begin{lem}\label{lem:lepar}
Let $M$ be a line-element parallelizable manifold. Then for any smooth functions $\{\lambda_i\}_{i=1}^{n}\subset C^{\infty}(M;\mathbb{R})$ which sum to zero, there
exist a smooth $g_0$-self-adjoint, traceless
endomorphism $H$ of $TM$ with $\{\lambda_i\}_{i=1}^{n}$ as its eigenvalues.
\end{lem}

\begin{proof}
Suppose that $TM=\oplus_{i=1}^n L_i$
and let $p \in M$. Then there
exists an open
neighborhood $U_p$ of $M$
on which there exists a (local) basis $(f_1,\cdots, f_n)$ of $TM$ with $\langle f_i\rangle_{\RR} =L_i$, for all $i=1,\cdots,n$.
Using Gram-Schmidt we
obtain on $U_p$ a local
$g_0$-orthonormal frame
$e:=(e_1,\cdots ,e_n)$ with
$\langle e_1,\cdots ,e_j\rangle_{\RR}=\oplus_{i=1}^j L_j$ for
all $j=1,\cdots,n$.
This allows us to define an 
$H$ on all of $U_p$, diagonal
with respect to $e$, with the required eigenvalues. 

 Now suppose that $q\in U_p\cap U_{\tilde p}$, where both $U_p$ and $U_{\tilde p}$ are as above. Then the definition of $H(q)$ is the same for both because $\langle e_j\rangle_{\RR}$ only depends on $\oplus_{i=1}^{j}L_i$ for all $j=1,\cdots,n$. 
\end{proof}

\section{Dimension three}
The purpose of this section to discuss the extent to which Theorem \ref{thm0} holds in dimension three. As we see in Lemma \ref{lem:3dplus}, the theorem cannot be true as stated, because the set of directions in which scalar curvature tends to $-\infty$ uniformly cannot be dense. However, as we see in Lemma \ref{3dnonempty}, the set is non-empty, with at least one interior point assuming orientability.

Both of these results rely on the ability to find an orthonormal frame $\{e_1,e_2,e_3\}$ (with respect to a fixed Riemannian metric $g_0$) so that $g_0([e_1,e_2],e_3)$ has certain properties. To this end, we fix a $g_0$-orthonormal frame $\{\hat{e}_1,\hat{e}_2,\hat{e}_3\}$ on some open neighbourhood $U\subseteq M$, and note that other orthonormal frames can be written as 
\begin{align}\label{3dbasischange}
e_i&=\sum_{j=1}^{3}R_{ij}\hat{e}_j, \ \text{for all} \ i=1,2,3,
\end{align}
where $R=(R_{ij}):U\to SO(3)$ is smooth.
We have 
\begin{align}\label{3dframechange}
\begin{split}
g_0([e_1,e_2],e_3)&=\sum_{i,j,k=1}^{3}g_0([R_{1i}\hat{e}_i,R_{2j}\hat{e}_j],R_{3k}\hat{e}_k)\\
&=\sum_{i,j,k=1}^{3}g_0(R_{1i}\hat{e}_i(R_{2j})\hat{e}_j-R_{2j}\hat{e}_j(R_{1i})\hat{e}_i,R_{3k}\hat{e}_k)+\textit{lots}\\
&=\sum_{i,j=1}^{3}R_{1i}\hat{e}_i(R_{2j})R_{3j}-\sum_{i,j=1}^{3}R_{2j}\hat{e}_j(R_{1i})R_{3i}+\textit{lots},
\end{split}
\end{align}
where $\textit{lots}$ denotes terms involving no derivatives of $R$. 

We start with Lemma \ref{lem:3dplus}; this shows that the relevant set of directions cannot be dense, and relies heavily on the formula for the scalar curvature that applies when the multiplicity is $(1,1,1)$.

\begin{lem}\label{lem:3dplus}
   Let $(M,g_0)$ be a
   three-dimensional Riemannian manifold. Then, there
   exists $H_0 \in \Gamma(\Sym_0(TM))$ such that
    any $H  \in \Gamma(\Sym_0(TM))$ nearby
  admits a point $p_H$ with $\lim_{t\to \infty}R(g_t)_{p_H}=+\infty$. 
\end{lem}

\begin{proof}
Choose a co-ordinate ball $B$ with geodesic normal co-ordinates $(x_1,x_2,x_3)$, with origin $p_{H_0}$. We construct $H_0$ as follows: \begin{itemize}
    \item 
    Apply Gram–Schmidt 
    to $\{\partial_{x_1},\partial_{x_2},\partial_{x_3}\}$
    to construct a $g_0$-orthonormal frame $\{\hat{e}_1,\hat{e}_2,\hat{e}_3\}$ on $B$. The condition that the co-ordinates are geodesic normal implies that $g_0(\partial_{x_i},\partial_{x_j})=\delta_{ij}+O(\left|x\right|^2)$, so $\hat{e}_i=\partial_{x_i}+O(\left|x\right|^2)$, implying that $[\hat{e}_i,\hat{e}_j]=0$ at $p_{H_0}$ for all $i,j\in \{1,2,3\}$. 
    \item Choose a $g_0$-orthonormal frame $\{e_1,e_2,e_3\}$ according to \eqref{3dbasischange} for some matrix function $R:B\to SO(3)$ so that $R(p_{H_0})=I$, and the first derivatives vanish at $p_{H_0}$. Then $g_0([e_1,e_2],e_3)=0$ at $p_{H_0}$. By applying $e_3$ to \eqref{3dframechange}, we see that we can alter the second derivatives of $R$ to insist that $e_3 g_0([e_1,e_2],e_3)=1$ at $p_{H_0}$.
    \item Choose smooth functions $\lambda_1,\lambda_2$ so that $e_1e_1\lambda_1>C$ (for a positive constant $C$ is to be specified later), $e_1\lambda_2=x_2$, and $e_1\lambda_1=0$ along $\{x_1=0\}$.
    \item Define $\lambda_3=-\lambda_1-\lambda_2$. Add constants to $\lambda_1$ and $\lambda_2$ so that $\lambda_1<0<\lambda_2<\lambda_3$.
\end{itemize}  
These conditions imply that the zero level sets of the three functions $g_0([e_1,e_2],e_3)$, $e_1\lambda_1$ and $e_1\lambda_2$ intersect transversally at $p_{H_0}$. 
By making $C$ large, we can ensure that $F_1^{(1)}$ is large at the origin. This then implies that $\lim_{t\to \infty}R(g_t)=+\infty$ at the origin.
The existence of such a point survives perturbations of the frame as well as the $\lambda_1,\lambda_2$ functions. But at such a point, we have $F_{12}^3=0=F_1^{(2)}$ and $F_1^{(1)}$ large; Corollary \ref{cor:scalfor} then implies the result. 
\end{proof}

We have established that the set of directions along which scalar curvature tends to $-\infty$ uniformly cannot be dense. Our next result addresses the issue of whether there can be \textit{any} directions with this good property, at least in case $M$ is orientable. 
\begin{lem}\label{3dnonempty}
Let $(M, g_0)$  be a compact and orientable three-dimensional Riemannian manifold.  Then, there exists an open, non-empty  subset $\Yod \subset T_{g_0}\Metmu$
 such that for every $h \in \Yod$ we have 
$ \lim_{t\to \infty}R(\gamma_h(t))=-\infty $ uniformly.
\end{lem}
\begin{proof}
By \cite{Sti}, orientable three-manifolds are parallelizable.
Fix a global reference $g_0$-orthonormal frame 
$\{\hat{e}_1,\hat{e}_2,\hat{e}_3\}$
and note that
any other smooth $g_0$-orthonormal frame $e=\{e_1,e_2,e_3\}$ with the same orientation is given by $e=R\hat{e}$ (as in \eqref{3dbasischange}), where $R:M\to SO(3)$ is smooth.
Consider the jet space $J^2(M,SO(3)\times \mathbb{R})$ (see
Definition 2.1 in Chapter 2 of \cite{GG}) and 
the closed subset $W\subset J^2(M,SO(3)\times \mathbb{R})$ defined by $W_1\cap W_2\cap W_3\cap W_4$, where $W_i$ are the equivalence classes of functions $(e,f)$ defined according to
\begin{align*}
    W_1&=\{e_1 f=0\},\\ W_2&=\{g_0([e_1,e_2],e_3)=0\},\\
    W_3&=\{g_0(e_1,\nabla_{g_0} (g_0([e_1,e_2],e_3)))\},\\ W_4&=\{g_0(e_1,\nabla_{g_0} (e_1 f))=0\}.
\end{align*}
 We show that $W$ is a submanifold of co-dimension four. Indeed, the first condition is a non-degenerate condition on the 1st-order jet of $f$ and the second condition is a non-degenerate condition on the 1st-order jet of $e$. To check the statement about the second equality, consider \eqref{3dframechange}, written locally. Differentiating $RR^T=I_3$ gives that, for each $j$, we have 
$$ 
(\hat{e}_j R) R^T+R (\hat{e}_j R)^T=0=(\hat{e}_j R)R^T+((\hat{e}_j R)R^T)^T.
$$
Therefore, $S_j:=(\hat{e}_j R)R^T$ is skew-symmetric for $j=1,2,3$. Note that the skew-symmetric matrices $S_j$ are all independent in the jet space. Writing in terms of $S$, we get $\hat{e}_j R=S_j R$, so 
\begin{align}
g_0([e_1,e_2],e_3)
&=
\sum_{i,j,k=1}^{3}R_{1i}(S_i)_{2k}R_{kj}R_{3j}-\sum_{i,j,k=1}^{3}R_{2j}(S_j)_{1k} R_{ki}R_{3i}+\textit{lots}\nonumber \\
&=
\sum_{i=1}^{3}R_{1i}(S_i)_{23}-\sum_{i=1}^{3}R_{2i}(S_i)_{13} +\textit{lots}
\end{align}
where \emph{lots} stands for
zeroth order terms of $R$.
Now the skew-symmetric $S_i$ terms are independent in the jet space, and at least one of the coefficient terms $R_{11},R_{12},R_{13}$ is non-zero, so we get transversality. 
The third condition is a second-order condition for $e$, and the fourth is a second order condition for $f$, so all of these co-dimension one submanifolds intersect transversally. This shows that $W$ is a submanifold of co-dimension four in the relevant jet space.

Since $W$ has co-dimension four, the Thom Transversality Theorem (Theorem 4.9 in Chapter 2 of \cite{GG}) implies that the set of functions $(e,f)$ which do not intersect 
$W=W_1\cap W_2\cap W_3\cap W_4$ is open and dense.
Similarly, $W_3\cap W_4$
defines a submanifold of co-dimension two in the jet space, so the set of functions $(e,f)$ which only intersect this submanifold at a finite number of pairwise disjoint copies of $\mathbb{S}^1$ is open and dense. 
Thus, it is possible to choose a frame $e$ and a smooth function $f$ so that  $S=\{e_1f=0\}\cap 
\{g_0([e_1,e_2],e_3)=0\}$ is a disjoint union of finitely-many copies of $\mathbb{S}^1$, and that at $S$, $e_1$ is nowhere tangent to $S$; otherwise that point would lie on $W$, a contradiction.

We now proceed to construct the required $\mathcal{Y}_{g_0}$. First, define a vector field $\partial_{\theta}$ on $M$ to be unit tangent to each of the $\mathbb{S}^1$ copies. Also define a vector field $n=e_1\times_{g_0}\partial_{\theta}$. Now consider the two-dimensional (oriented) submanifold $\tilde{S}$ found by evolving the one-dimensional submanifold $S$ by the short-time flow of $n$. Then we define the function $\lambda_2$ 
to be zero on $S$, and so that $e_1\lambda_2=1$ in a neighbourhood of $S$ in $\tilde{S}$; this is possible because $e_1$ is not tangent to $S$. Define $\lambda_1=f$, define $\lambda_3=-\lambda_1-\lambda_2$, and then add constants to $\lambda_1,\lambda_2,\lambda_3$ so that $\lambda_1<\lambda_2,\lambda_3$. Then the $H$ defined according to the eigenvalues $(\lambda_1,\lambda_2,\lambda_3)$ in the basis $\{e_1,e_2,e_3\}$, as well as all nearby choices of $H$, have scalar curvature tending to $-\infty$ uniformly along the Ebin geodesic. This follows from Corollary \ref{cor:scalfor}, using the fact that at least one of $g_0([e_1,e_2],e_3)$, $e_1\lambda_1$, $e_1\lambda_2$ is always non-zero.  
\end{proof}

\section{Symmetric and traceless matrices}

Let $(M,g_0)$ be a compact Riemannian manifold. According to Lemma \ref{lem:goodframe}, locally specifying a choice of $H\in\Gamma(\Sym_0(TM))$ is equivalent to specifying a clustering $(m_1,\cdots,m_L)$ for the eigenvalues of $H$, a $g_0$-orthonormal frame, smooth generalised eigenvalue functions $(\lambda_1,\cdots,\lambda_L)$, as well as the traceless block matrices $S_i$ which are small wherever there are $m_i$ eigenvalues close to $\lambda_i$, $i=1,...,L$. 
 However, the mutiplicity $(m_1,\cdots,m_L)$ and the choice of $g_0$-orthonormal frame are the only two of these four objects which influence the applicability of
Theorem \ref{curvatureasy}.

In Lemma \ref{lem:candidateopen}, we saw that it is convenient to let $(m_1,\cdots,m_L)$ actually \textit{coincide} with the eigenvalue multiplicities of $H$ at certain points $p$, so that $S_i$ vanishes at $p$. Therefore, in this section, we provide a pointwise study of symmetric matrices, their eigenvalue multiplicities, and how they are acted on by
the special orthogonal group
$SO(n)$. The theory developed in the next two sections builds towards establishing that, in the \textit{generic} situation, not all eigenvalue multiplicities can occur for our smooth sections $H$; this will eventually come in handy when proving
Theorem \ref{thm0}.

\subsection{Geometry of the ${\rm SO}(n)$ orbits}\label{sec:symtracefree}

We denote by $\Sym_0(n)$ the vector space of $n\times n$ symmetric traceless matrices with real entries. As a vector space, $\Sym_0(n)$ comes equipped with the natural inner product given by 
\begin{align*}
    \langle S_1,S_2\rangle =\tr(S_1\cdot S_2);
\end{align*}
this inner product gives us an isometry between $\Sym_0(n)$ and $\mathbb{R}^{\frac{n(n+1)}{2}-1}$. The special orthogonal group ${\rm SO}(n)$ acts isometrically on $\Sym_0(n)$ via
\begin{eqnarray*}
   (A,S) \mapsto A \cdot S \cdot A^T\,.
\end{eqnarray*}
Note that two $n\times n$ points in $\Sym_0(n)$ lie in the same ${\rm SO}(n)$ orbit if and only if they have the same real eigenvalues, occurring with the same multiplicities.

\begin{defin}
For  
$$
\lambda=(\lambda_1,...,\lambda_n) \in \RR^n_0:=\{ \lambda \in \RR^n: \sum_{i=1}^n \lambda_i =0\}
$$
we set
$$
  D_\lambda = {\rm Diag}(\lambda_1,..., \lambda_n).
$$
\end{defin}

 Thus, every orbit has a unique diagonal representative $D_\lambda$,
$\lambda \in \RR^n_0$, with diagonal entries 
\begin{eqnarray}
       \lambda_1 \leq \lambda_2 \leq \cdots \leq \lambda_n  \,.
\label{eigenvalueslambda}
\end{eqnarray}
We will denote the corresponding
${\rm SO}(n)$-orbit by $O_\lambda$; note that $\lambda\in \mathbb{R}^n_0$ depends continuously on the choice of orbit, but not smoothly. As a first consequence, the orbit space 
$\Sym_0(n)/{\rm SO}(n)$ is homeomorphic to the set
\begin{eqnarray*}
   \overline{\Sym}_0(n) &:=& \big\{ \lambda \in \RR^n_0 \mid \lambda_1 \leq \lambda_2 \leq \cdots \leq \lambda_n \big\}\,,
\end{eqnarray*}
equipped with the induced topology. 
Next, we are going to describe orbits and their normal spaces more precisely.
For a given $S\in \Sym_0(n)$, we denote by 
$$
 \mathfrak{m}=(m_1, \cdots, m_L)
 $$ 
 the multiplicities of the eigenvalues of $S$ ordered according to \eqref{eigenvalueslambda}. Note that
 $L\geq 1$, $m_1,...,m_L\geq 1$, $\sum_{i=1}^L m_i =n$. 
Setting 
\begin{align*}
    M_1:=m_1, \qquad M_2:=m_1+m_2,  \qquad \cdots, \qquad  M_L=m_1+\cdots +m_L=n,
\end{align*}
we then have 
\begin{align*}
    \lambda_1=\cdots = \lambda_{M_1} <\lambda_{M_1+1}=\cdots =\lambda_{M_2} <  \cdots 
    <\lambda_{M_{L-1}+1}=\cdots =\lambda_n.
\end{align*}
Notice that the multiplicity $\m_0=(n)$ only occurs for the zero matrix
and that $\m=(1,1,...,1)$ is the generic
case, meaning that all eigenvalues are pairwise distinct.

We denote by $I_{m_i} \in {\rm O}(m_i)$ the identity matrix.

\begin{lem}          \label{lem:normalorbit}
Suppose that $\lambda\in \RR^n_0$ satisfies {\rm (\ref{eigenvalueslambda})} and let
$\mathfrak{m}=(m_1,...,m_L)$ be the eigenvalue multiplicities of $D_\lambda$. Then, $\Sym_{0}(n)$ has the orthogonal decomposition \begin{eqnarray*}
    \Sym_0(n)
  &=&
  \left\{ \left( \begin{array}{ccc} S_1 && \\ & \ddots & \\ && S_L \end{array}\right) \mid
            S_i \in \Sym_0(m_i), \ i=1,...,L\right\} \nonumber \\
      &&
          \oplus  \left\{ \left( \begin{array}{ccc} \mu_1 \cdot I_{m_1} && \\ & \ddots & \\ && \mu_L\cdot I_{m_L}
          \end{array}\right) 
          \mid  \sum_{j=1}^L m_j\mu_j  =0\right\}
          \\
      &&
      \oplus  \left\{ S(R_{12},\cdots,R_{(L-1)L})
          \mid   R_{ij}\in {\rm Mat}(m_i,m_j), \ 1\le i<j\le L \right\} \nonumber
\end{eqnarray*}
with
$$
 S(R_{12},\cdots ,R_{(L-1)L})
 =
 \left( \begin{array}{cccc} 0&R_{12}&\cdots &R_{1L}\\
      (R_{12})^T & 0&\cdots & \vdots \\
      \vdots &\ddots& \ddots & R_{(L-1)L}\\
      (R_{1L})^T&\cdots&(R_{(L-1)L})^T&0
          \end{array}\right) \,.
$$
Furthermore, the homogeneous space
$O_\lambda$ has the representation
\begin{eqnarray*}
      O_\lambda ={\rm SO}(n)/ {\rm SO}_{\m}(n)
\end{eqnarray*}     
with 
\begin{eqnarray}\label{eqn:SOm}
     {\rm SO}_{\m}(n)=
       {\rm S}\big({\rm O}(m_1)\times  {\rm O}(m_2)\times  \cdots \times {\rm O}(m_L)\big)\,,  
\end{eqnarray}
and $T_{D_{\lambda}}O_{\lambda}$ coincides with the third summand in the expression for $\Sym_0(n)$. 
\end{lem}

\begin{proof}
The fact that the decomposition is orthogonal with respect to our Euclidean inner product $\langle S_1,S_2\rangle =\tr(S_1\cdot S_2)$ is obvious. We now compute $T_{D_{\lambda}}O_{\lambda}$. To this end, let $A(t)$, $t \in (-\epsilon,\epsilon)$, be a smooth curve in ${\text{SO}}(n)$ with $A(0)=I_n$ and $A'(0)=T\in \text{Scew}(n)=\mathfrak{so}(n)$. Then 
$$
 \tfrac{d}{dt}\vert_{t=0}\,\, A(t) \cdot D_\lambda \cdot A(t)^T = (T \cdot D_\lambda -D_\lambda \cdot T)\,.
$$
Writing
\begin{align*}
   T=     \left( \begin{array}{cccc} T_{11}&R_{12}&\cdots &R_{1L}\\
      -(R_{12})^T & T_{22}&\cdots & \vdots\\
      \vdots &\ddots& \ddots &R_{(L-1)L}\\
     - (R_{1L})^T&\cdots&-(R_{(L-1)L})^T&T_{LL}
          \end{array}\right) 
\end{align*}
with $T_{ii} \in \Scew(m_i)$ and $R_{ij} \in {\rm Mat}(m_i,m_j)$
we deduce that $\tfrac{d}{dt} \,\,A(t) \cdot D_\lambda \cdot A(t)^T$ at $t=0$ is given by
a matrix where $T_{11},\cdots, T_{LL}$
is replaced by zero and
$R_{kl}$, $1\leq k < l\leq L$
is replaced by $(\lambda_{M_l}-\lambda_{M_k})\cdot R_{kl}$.
Since $R_{ij}$ can be chosen arbitrarily, and $\lambda_{M_l}\neq \lambda_{M_k}$ for all $k\neq l$, the claim about $T_{D_{\lambda}}\Sym_0(n)$ follows. It is clear that the isotropy subgroup of ${\rm SO}(n)$ for the point $D_\lambda$ is $
       {\rm SO}_{\m}(n)$ so that $O_{\lambda}$ has the required expression as a homogeneous space.  
\end{proof}

An orbit $O_\lambda$ is a principal orbit of the action of ${\rm SO}(n)$ in $\Sym_0(n)$ if all eigenvalues are pairwise distinct, or equivalently,  $O_\lambda$ lies in the interior of the orbit space.
We call principal orbits \textit{regular} and non-principal orbits \textit{singular}. Notice that the boundary
of the orbit space is just the set of singular orbits.

\subsection{Stratification of the orbit space}

\begin{defin} 
 Let $L\geq 1$, $m_1,...,m_L\geq 1$ be integers with $\sum_{i=1} ^L m_i=n$
and let $\mathfrak{m}=(m_1,...,m_L)$. Then   we define
\begin{eqnarray*}
      \overline{F} (\m) &:=& \{O_\lambda \mid   \lambda_1 = \cdots =\lambda_{m_1}< \lambda_{m_1+1} \cdots
                         < \lambda_{M_{L-1}+1} = \cdots = \lambda_n\}\,.
\end{eqnarray*}
\end{defin}

The face $\overline{F}(\m)$ contains all the orbits with the same eigenvalue multiplicities $m_1,...,m_L$.
The dimension of $\overline{F} (\m)$ equals $L-1$ in the orbit space.
If $\m_p=(1,...,1)$, then the principal face $\overline{F} ({\m_p})$, an open subset of $\overline{\Sym}_0(n)$,
contains all principal orbits.  The faces distinct from the principal face are called singular faces. The following definition provides a way of describing which faces are `more singular' than others.

\begin{defin}
Let $\m=(m_1,...,m_L)$ and $\tilde \m=(\tilde m_1,...,\tilde m_{\tilde L})$
be two eigenvalue multiplicities. Then we write $\m \leq \tilde \m$ if and only if
there is a surjective and non-decreasing function $J:\{1,\cdots,\tilde{L}\}\to 
\{1,\cdots,L\}$ so that $m_i=\sum_{j\in J^{-1}(i)}\tilde{m}_j$ for all $i\in 
\{1,\cdots,L\}$. 
We write $\m < \tilde \m$ if $\m \leq \tilde \m$ and $\m \neq \tilde \m$.
\end{defin}

Note that if $\mathfrak{m}_0=(n)$, then $\mathfrak{m}_0\le \tilde{\mathfrak{m}}$ for any multiplicity $\tilde{\mathfrak{m}}$, so the corresponding face $\overline{F}((n))$ in $\overline{\Sym}_0(n)$ can be thought of as `the most singular face'. This is reflective of the fact that the origin in $\Sym_0(n)$ is the only zero-dimensional orbit. 

For a subset $\overline{X} \in \overline{\Sym}_0(n)$ we denote by
$\cl(\overline{ X})$ its closure. The next lemma establishes a relationship between the closure of a particular face $\overline{F}(\tilde{\mathfrak{m}})$, and the faces $\overline{F}(\mathfrak{m})$ that are `more singular'. 

\begin{lem}\label{lem:closure}
 Let $\m, \tilde \m$ be two eigenvalue multiplicities. Then $\m \leq \tilde \m$
if and only if  $\overline{F}({\m})$ is contained in the closure
of $\overline{F}(\tilde {\m})$. Moreover
\begin{eqnarray*}
                \cl(\overline{F}(\tilde \m))\backslash \overline{F}(\tilde \m)&=& \bigcup_{\m < \tilde \m} \overline{F}(\m)\,.
\end{eqnarray*}
\end{lem}

\begin{proof}
Take two eigenvalue multiplicities $\mathfrak{m}=(m_1,\cdots\!,m_L)$ and $\tilde{\mathfrak{m}}=(\tilde{m}_1,\cdots\!,\tilde{m}_{\tilde{L}})$. If $\mathfrak{m}\le \tilde{\mathfrak{m}}$, take an arbitrary $D_{\lambda}\in \overline{\Sym}_0(n)$ so that $\lambda$ has $\mathfrak{m}$ as its multiplicity vector, thus
\begin{align*}
    D_{\lambda}=\text{diag}(\underbrace{\lambda_1,\cdots,\lambda_1}_{m_1},\underbrace{\lambda_2,\cdots,\lambda_2}_{m_2},\cdots, \underbrace{\lambda_L,\cdots,\lambda_L}_{m_L})
\end{align*}
for some numbers $\lambda_1,\cdots,\lambda_L$ which are strictly increasing. It is clear that $D_{\lambda}$ can be approximated by diagonal matrices of the form 
\begin{align*}
    \begin{pmatrix}
      \lambda_1 I_{m_1}+S_{m_1} & 0&\cdots &0\\
      0&\lambda_2 I_{m_2}+S_{m_2}&\cdots&0\\ 
      \vdots&\vdots&\ddots&\vdots\\
      0&0&\cdots&\lambda_{L}I_{m_L}+S_{m_L}
    \end{pmatrix}
\end{align*}
for traceless symmetric matrices $S_{m_i} \neq 0$, $i=1, \cdots, L$. Since $\mathfrak{m}\le \tilde{\mathfrak{m}}$, we can choose the $S_{m_i}$ matrices so that our approximating sequence lies in $\overline{F}(\tilde{\mathfrak{m}})$. Thus,
$\overline{F} (\m) \subseteq \cl(\overline{F}(\tilde \m))$.

On the other hand,
suppose $\overline{F} (\m) \subseteq \cl(\overline{F}(\tilde \m))$. Then, 
if it is \textit{not} true that $\mathfrak{m}\le \tilde{\mathfrak{m}}$, then for any $D_{\lambda}\in \overline{F}(\mathfrak{m})$ and $D_{\tilde{\lambda}}\in \overline{F}(\tilde{\mathfrak{m}})$, there is an $i\in \{1,\cdots,L\}$ and a $j\in \{1,\cdots,\tilde{L}\}$ so that the 
$\tilde{\lambda}_j$ eigenspace of
$D_{\tilde{\lambda}}$  has a non-trivial intersection with both the $\lambda_i$ and $\lambda_{i+1}$ eigenspaces of $D_{\lambda}$. Therefore, since $\lambda_i<\lambda_{i+1}$, $D_{\lambda}$ cannot be approximated by diagonal matrices in $\overline{F}(\tilde{\mathfrak{m}})$.

Now for a given $\tilde{\mathfrak{m}}=(\tilde{m}_1,\cdots ,\tilde{m}_{\tilde{L}})$, the set $\cl(\overline{F}(\tilde \m))\backslash \overline{F}(\tilde \m)$ consists precisely of the matrices 
\begin{align*}
    D_{\tilde{\lambda}}=\text{diag}(\underbrace{\tilde{\lambda}_1,\cdots,\tilde{\lambda}_1}_{\tilde{m}_1},\underbrace{\tilde{\lambda}_2,\cdots,\tilde{\lambda}_2}_{\tilde{m}_2},\cdots, \underbrace{\tilde{\lambda}_L,\cdots,\tilde{\lambda}_L}_{\tilde{m}_L})
\end{align*}
with $\tilde{\lambda}_1\le \tilde{\lambda}_2\le \cdots \le \tilde{\lambda}_{\tilde{L}}$, with at least two of these $\tilde{L}$ eigenvalues being equal. By ranging over the possible choices for eigenvalues being equal, we get all possible matrices in $\overline{F}(\mathfrak{m})$, for $\mathfrak{m}<\tilde{\mathfrak{m}}$. 
\end{proof}

\subsection{Stratification of $\Sym_0(n)$}

It becomes important for us to relate the faces $\overline{F}(\mathfrak{m})\subseteq \overline{\Sym}_0(n)$ with their  counterparts in $\Sym_0(n)$. To this end, denote by
\begin{eqnarray*}
  \pi : \Sym_0(n) \to   \overline{\Sym}_0(n)= \Sym_0(n)/{\rm SO}(n)
\end{eqnarray*}
the projection to the orbit space. Let
$S(\Sym_0(n))$  denote the unit sphere
in $\Sym_0(n)$ and set $\bar S:=\pi ( S(\Sym_0(n)))$ and 
$\overline{S}(\m)=\overline{F}(\m) \cap \bar S$. Note that $\overline{S}((n))=\emptyset$
but $\overline{S}(\m)\neq \emptyset$
for all  $\m>\m_0=(n)$.

\begin{defin}\label{definF}
 Let $L\geq 1$, $m_1,...,m_L\geq 1$ be integers with $\sum_{i=1} ^L m_i=n$
and let $\mathfrak{m}=(m_1,...,m_L)$. 
 We set
$$
  F(\m) := \pi^{-1} (\overline{F}(\m))\quad\textrm{ and }\quad S(\m):=F(\m) \cap S(\Sym_0(n))\,.
$$
\end{defin}

It is important to notice that if $L=1$, then this face $F((n))$ 
is just one point, the origin.
If $L=2$, then a face $F(\m_k)$
is  a cone over an orbit
with eigenvalue multiplicity $\m_k=(k,n-k)$ for some $1\leq k \leq n-1$.
The corresponding
spherical faces $\bS(\m_k)$ are the vertices of  $\bS$.
More generally, $\bF(\m)$ is a cone over $\bS(\m)$, if $\bF(\m)$ has positive dimension. In particular
the orbit space itself is a cone over 
$\bar S$.

\begin{lem}\label{lem:submanifold}
  Let $\m=(m_1,\cdots ,m_L)$ with $m_1,\cdots ,m_L \geq 1$ and $\sum_{i=1}^L m_i=n$.
Then $F(\m)$ is an embedded
submanifold of $\Sym_0(n)$ of
$\codim F(\m)=\sum_{i=1}^{L}d(m_i)$,
where for each $i=1,\cdots ,L$ we set
$d(m_i)=\tfrac{m_i(m_i+1)}{2}-1$.
 Moreover, if 
$L\geq 2$, $F(\m)$ is
a cone over the spherical face $S(\m)$
and Lemma {\rm \ref{lem:closure}} 
 holds  true for the faces $F(\m)$. For a diagonal matrix $D_\lambda \in F(\m)$ the normal space $\nu_{D_\lambda} F(\m)$ is
given by
           \begin{eqnarray}\label{normalspace}
             \nu_{D_\lambda}F(\m) &=&
        \left\{ \left( \begin{array}{ccc} S_1 && \\ & \ddots & \\ && S_L \end{array}\right) \mid
            S_i \in \Sym_0(m_i)\,\,\,,i=1,...,L\right\} \,.
\end{eqnarray}
\end{lem}

\begin{proof}
If $\m=(n)$, that is $L=1$, the claim is obvious. So let us assume $L\geq 2$
and let $S \in F(\m)$. Using the group action of ${\rm SO}(n)$
we may assume  $S=D_\lambda$. Equality \eqref{normalspace} is a consequence of Lemma \ref{lem:normalorbit}.
We denote by ${\rm Fl}$ the space of diagonal traceless symmetric matrices. Notice that there is a vector space isomorphism between 
${\rm Fl}$ and $ \RR^{n-1}$.
Next let ${\rm Fl}_\lambda$ denote the unique 
$(L-1)$-dimensional subspace of
${\rm Fl}$ containing all diagonal traceless matrices
having the same eigenspaces as $D_\lambda$. We denote by
$B_\epsilon^{{\rm Fl}}(D_\lambda)$ the open
$\epsilon$-ball in ${\rm Fl}$ and by
$B_\epsilon^{{\rm Fl}_\lambda}(D_\lambda)$
the open $\epsilon$-ball in ${\rm Fl}_\lambda$, centered at $D_\lambda$, with respect to the flat metrics on ${\rm Fl}$ and ${\rm Fl}_\lambda$,
respectively. Notice that
$B_\epsilon^{{\rm Fl}}(D_\lambda) \cap F(\m)=
B_\epsilon^{{\rm Fl}_\lambda}(D_\lambda)$ by the very definition
of $F(\m)$ (Definition \ref{definF}). If $\epsilon$ is small enough,
then for all $D_{\lambda '}  \in B_\epsilon^{{\rm Fl}_\lambda}(D_\lambda) $ the isotropy group
is given by  ${\rm SO}_{\m}(n)$, see \eqref{eqn:SOm}.
Since ${\rm SO}(n)$ acts polarly (see proof of Lemma \ref{lem:normalorbit}), 
each orbit intersects ${\rm Fl}$ orthogonally, and consequently the
set ${\rm SO}(n)\cdot B_\epsilon^{{\rm Fl}_\lambda}(D_\lambda)$
is an embedded submanifold of $\Sym_0(n)$
of dimension
 $$
 \dim (O_\lambda) +L-1=\tfrac{1}{2}n(n-1)-( \sum_{i=1}^L \tfrac{1}{2}m_i(m_i-1))+L-1\,.
$$ 
Notice that
$$
  \tfrac{1}{2}n(n+1)-1 - \big(\dim (O_\lambda) +L-1 \big) 
  =\sum_{i=1}^l d(m_i)\,.
$$
It remains to show that 
${\rm SO}(n)\cdot B_\epsilon^{{\rm Fl}_\lambda}(D_\lambda)$ is an open subset of $F(\m)$.
To see this let us denote by $B_\epsilon(D_\lambda)$ 
the open $\epsilon$-ball in  $\Sym_0(n)$ with center $D_\lambda$. We claim
that for small $\epsilon'<\epsilon$ we have
$$
 B_{\epsilon'}(D_\lambda)\cap F(\m)=
B_{\epsilon'}(D_\lambda) \cap 
\big( {\rm SO}(n)\cdot B_{\epsilon'}^{{\rm Fl}_\lambda}(D_\lambda)\big)\,.
$$ 
Suppose on the contrary that there
would exist a sequence $(S_k)_{k \in \NN} \in F(\m)$ with
$\lim_{k\to \infty} S_k=D_\lambda$
but $S_k \not\in  {\rm SO}(n)\cdot B_{\epsilon}^{{\rm Fl}_\lambda}(D_\lambda)$ 
for all $k \in \NN$. 
Using the ${\rm SO}(n)$-action, we can find $A_k\in {\rm SO}(n)$ converging to the identity so that the matrices
 $S_k':= A_k \cdot  S_k \cdot A_k^T$ converge to $D_{\lambda}$, and satisfy $S_k'-D_{\lambda}\in \nu_{D_{\lambda}}F(\mathfrak{m})$. Then we can find $B_k\in {\rm SO}_{\mathfrak{m}}(n)$ so that $\hat{S}_k:=B_k\cdot S_k'\cdot B_k^{T}$ is diagonal and $\lim_{k\to \infty}\hat{S}_k-D_{\lambda}=\lim_{k\to \infty}B_k\cdot(S_k'-D_{\lambda})\cdot B_k^T=0$.
 By assumption we cannot have
 $\hat S_k \in B_{\epsilon}^{{\rm Fl}_\lambda }(D_\lambda)$ for all $k \in \NN$. 
 But that implies that
 $\hat S_k$ must have at least $L+1$ pairwise
 distinct eigenvalues, contradicting the definition
 of $F(\m)$. This shows the first claim. The second claim is obvious.
\end{proof}
\begin{cor}\label{cor:normalspace}
   For any $S$ sufficiently close to $D_{\lambda}$, there exists $A$ close to the identity so that $A\cdot S\cdot A^{T}\in \nu_{D_{\lambda}}F(\mathfrak{m})$.
\end{cor}

\subsection{Low-dimensional cases}\label{sec:Casen2}

In case $n=2$,
the function $H:\mathbb{R}^2\to \Sym_0(2)=\RR^2$ with 
 \begin{eqnarray*}
  H(a,b):=\left( \begin{array}{cc} a & b \\ b &-a \end{array}\right)\,
\end{eqnarray*}
is a diffeomorphism. The norm of the matrix $H(a,b)$ in $\Sym_0(2)$ is 
$\sqrt{2(a^2+b^2)}$. We write $(a,b)=r\cdot (\cos(\varphi),\sin(\varphi))$
for $\varphi \in [0,2\pi)$ and $r=\sqrt{a^2+b ^2} $. 
The eigenvalues at $H(a,b)\in \RR^2$ are $r,-r$;
if $r>0$ and $\varphi \in (0,\pi) \cup (\pi, 2\pi)$, then the corresponding eigenvectors are given by 
\begin{align*}
   (1+\cos(\varphi),\sin(\varphi)), \qquad  (1-\cos(\varphi),\sin(\varphi)),
\end{align*}
while if $\varphi=0,\pi$, eigenvectors are given by $(1,0)$ and $(0,1)$. 
It follows that we cannot choose the eigenvectors at $(0,0)$ in a continuous manner. Even more interestingly, we cannot continuously choose the eigenvectors
on $\RR^2\backslash \{(0,0)\}$ globally. 

All possible eigenvalue multiplicities are $\m_2=(2)$ and $\m_0=(1,1)$.
The origin is a \mbox{$0$-dimensional} face, corresponding to $\m_2$. All other (principal) orbits are distance spheres, $\Sym_0(2)/{\rm SO}(2)=[0,\infty)$ and 
$\overline{S}(\m_1)$ equals to $\bS=\{\pi(\text{diag}(\frac{1}{\sqrt{2}},-\frac{1}{\sqrt{2}}))\}$. The eigenvalues themselves do depend continuously on $(a,b)\in \mathbb{R}^2$, but not smoothly at $(0,0)$. 

In case $n=3$, $\Sym_0(3)=\RR^5$, and all possible eigenvalue multiplicities are
\begin{align*}
    \m_5=(3), \qquad \m_2^1=(2,1), \qquad \m_2^2=(1,2),  \qquad \text{and} \ \m_0=(1,1,1). 
\end{align*}
The origin in $\Sym_0(3)=\RR^5$ is a \mbox{$0$-dimensional} face, corresponding to $\m_5$.
The spherical face $S(\m_0)$ is the principal face in $S(\Sym_0(3))$. Notice that the action of
${\rm SO}(3)$ of on $S(\Sym_0(3))$
is cohomogeneity one and the principal isotropy group is
 $\mathbb{Z}_2\times \mathbb{Z}_2$ (diagonal matrices with diagonal elements $\pm 1$), 
 so the principal orbits are diffeomorphic to ${\rm SO(3)}/\mathbb{Z}_2\times \mathbb{Z}_2$. 
On the singular orbit $S(\m_2^1)=\RP^2$ the first two eigenvalues agree,
 whereas  on  the singular orbit 
 $S(\m_2^2)=\RP^2$ the last two eigenvalues  are the same.
The faces $S(\m_2^1)$ and  $S(\m_2^2)$ are in the closure of the principal face $S(\m_4)$.
In dimension three there are precisely two singular faces of positive dimension, each
being a cone over a singular orbit. 

If $n=4$, all possible eigenvalue multiplicities are
\begin{align*}
  \m_8=(4), \qquad \m_5^1=(3,1), \qquad \m_5^2=(1,3), \qquad  \m_4=(2,2),\\
  \m_2^1=(2,1,1), \qquad \m_2^2=(1,2,1), \qquad \m_2^3=(1,1,2), \qquad  \m_0=(1,1,1,1).
\end{align*}
The space $\Sym_0(4)$ is $9$-dimensional
and the  orbit space  $\Sym_0(4)/{\rm SO}(4)$
has dimension $3$, thus $\bar S$ has
dimension $2$; it is a spherical triangle
with three edges corresponding to
$\mathfrak{m}_2^i$, $i=1,2,3$. The intersection points of these edges 
 correspond to $\mathfrak{m}_5^1$, $\mathfrak{m}_5^2$, and $\mathfrak{m}_4$, respectively. The interior of $\bar S$
corresponds to $\m_0$, of course.

\section{Riemannian deformations which preserve the volume form}
\label{sec:preliminaries}

As before, we fix a smooth  Riemannian metric $g_0 \in \Met$ on $M$ and denote by $\mu := \mu_{g_0}$ its volume density. 
In Definition \ref{def:globalgen}
we will define generic volume-preserving deformations $h \in T_{g_0}
\Metmu$. Considering $h$ as an
endomorphism $H \in \Sym_0(TM)$ (see \eqref{eqn:H}) it is sufficient to define generic 
endomorphisms $H \in \Sym_0(TM)$.
This is done by defining the singular 
set $\Sing(M) \subset \Sym_0(TM)$
in Definition \ref{def:SingM},
a compact Whitney stratified subspace. We subsequently define generic endomorphisms as those which always intersect $\Sing(M)$ transversally.
In Theorem \ref{thm:generic}, we will then show that the set of generic deformations is open and dense
in the $C^\infty$-topology. This is to be expected, since the work of Trotman \cite{Tro1} implies that the set of
smooth maps $F:X\to Y$ which transversally intersect a closed Whitney stratified space 
$W \subset Y$ are open and dense.

\subsection{The singular set $\Sing(n)$}

In order to define the singular
set $\Sing(M)$ in Subsection \ref{subsec:SingM} we define 
the  singular set $\Sing(n)\subset \Sym_0(n)$
and show that it is a closed Whitney stratified subspace.

\begin{defin}
Let $n\geq 2$. We denote by
$$
 \Sing(n):= \bigcup\,\,\,
 \big\{ F(\m):\m=(m_1,\cdots ,m_L), \,\,L<n\big\}
$$
the \emph{singular set} of $\Sym_0(n)$.
\end{defin}

By the very definition of the strata
$F(\m)$ (Definition \ref{definF}), the set $\Sing(n)$ consists
precisely of those symmetric and traceless matrices which do not have pairwise
distinct eigenvalues. Notice now,
that by Lemma \ref{lem:submanifold}
$$
   \Sing(n)= \cl(F(\m_0))\backslash F(\m_0)\,,
$$
where $\mathfrak{m}_0=(1,\cdots,1)$ corresponds to the (open) principal face (of codimension $0$).
As a consequence,
the singular set $\Sing(n)$ is a closed subset of $\Sym_0(n)$ with a stratification
given by the singular faces $F(\m)$. Moreover, singular faces have codimension at least two, equality achieved for 
the eigenvalue multiplicities of type 
$$
\m_2^1=(2,1,\cdots ,1),\dots,\m_2^{n-1}(1,\cdots ,1,2).
 $$
Recall the following from \cite{Tro2} (for example): 

\begin{defin}\label{WTD}
Let $X,Y$ be two locally closed and disjoint submanifolds in $\mathbb{R}^m$ of dimensions $i\geq 0$ and $j\geq 0$, respectively. 
 Whitney's transversality conditions for  $X$ and $Y$ are the following:
\begin{itemize}
   \item[(a)] 
   If 
   $\{x_k\}_{k=1}^{\infty}\subset X$ has $\lim_{k\to \infty}x_k=y\in Y$ and $T_{x_k}X$ converges to an $i$-dimensional subspace $T$, then $T_y Y\subseteq T$ (in particular, $j\le i$).
     \item[(b)] 
     If $
     \{x_k\}_{k=1}^{\infty}\subset X$ and $\{y_k\}_{k=1}^{\infty}\subset Y$ satisfy 
    \begin{itemize}
        \item[$\bullet$] $\lim_{k\to \infty}x_k=\lim_{k\to \infty}y_k=y\in Y$,
        \item[$\bullet$] the secant lines $L_k$ between $x_k$ and $y_k$ converge to a line $L$,
        \item[$\bullet$] $T_{x_k}X$ converges to an $i$-dimensional subspace $T$,
    \end{itemize}
    then $L$ is contained in $T$. 
\end{itemize}
\end{defin}

We now show that faces $F(\m)$ and $F(\tilde \m)$ with $\m < \tilde \m$,
called adjacent faces,
satisfy Whitney's transversality condition (a).
To this end, we first give a related definition. 
\begin{defin} 
We say that the subspace $V\subset \Sym_0(n)$ intersects the face $F(\m)$ in a point
$S\in F(\mathfrak{m})$ transversally if $\Sym_0(n)=T_S F(\mathfrak{m})+V$. 
\end{defin}

We denote by $B_\epsilon(S)$ the open ball in $\Sym_0(n)$ of radius $\epsilon>0$ around $S$.

\begin{prop}                    \label{prop:transfaces}
Let $F(\m)$ be a singular face 
with $F(\m) \subset \cl(F(\tilde \m))$.  Furthermore,
let  $V \subset \Sym_0(n)$ be 
a subspace, which intersects $F(\m)$
transversally at a diagonal matrix $D_\lambda \in F(\m)$. Then there exists
$\epsilon_V>0$ such that 
$V$ intersects $F(\tilde \m)$ transversally on $B_{\epsilon_V}(D_\lambda) \cap F(\tilde \m)$.
\end{prop}

\begin{proof}
Let $D_\lambda \in F(\m)$. Firstly,
we consider diagonal elements $D_{\tilde \lambda}\in F(\tilde \m) \cap {\rm Fl}$ 
which are close to
$D_\lambda$; we do not insist
on an ordering of the eigenvalues.
Notice that
the eigenvalues $(\tilde \lambda_1,...,\tilde \lambda_n)$ of
$D_{\tilde \lambda}$ 
still inherit the
eigenvalue clustering of $\lambda$
in the following sense:
we have $\tilde \lambda_{i}< \tilde \lambda_{j}$
for all $i \in (1,m_1)$
and all $j\in (m_1+1,m_1+m_2)$
and so on. 
Notice now that Lemma
\ref{lem:submanifold} also holds for
$D_{\tilde \lambda}$. This then
shows that 
\begin{align}\label{normalfacecontainment}
\nu_{D_{\tilde \lambda}}F(\tilde \m)\subseteq \nu_{D_{\lambda}}F(\mathfrak{m}).
\end{align}
Clearly, the claim
follows in this case.
For arbitrary 
points $S \in B_{\epsilon_V}(D_\lambda) 
\cap F(\tilde \m) $, we can find $A\in  {\rm SO}(n)$ close to the identity so that $S':=ASA^{T}\in \nu_{D_{\lambda}}(\mathfrak{m})$ (Corollary \ref{cor:normalspace}).
Now since $ {\rm SO}_{\mathfrak{m}}(n)$ is the isotropy of $D_\lambda$, we can find $B\in {\rm SO}_{\mathfrak{m}}(n)$ so that
$B S'B^{T}=D_{\tilde{\lambda}}\in F(\tilde{\mathfrak{m}})\cap {\rm Fl}$ for some $B\in {\rm SO}_{\mathfrak{m}}(n)$. Note that $B$ is not necessarily close to the identity, but $D_{\tilde{\lambda}}$ is close to $D_{\lambda}$. Containment \eqref{normalfacecontainment} follows. Since $\nu_{D_{\lambda}}F(\mathfrak{m})$ is invariant under the action of ${\rm SO}_{\mathfrak{m}}(n)$, we find that $\nu_{S'}F(\tilde \m)\subseteq \nu_{D_{\lambda}}F(\mathfrak{m})$.
The claim then follows from the fact that $\nu_{S}F(\tilde \m)$ is close to $\nu_{S'}F(\tilde \m)$ (since $A$ is close to the identity). 
\end{proof}

As a direct consequence of Proposition \ref{prop:transfaces}, we find that 
 adjacent faces $F(\mathfrak{m})$ and $F(\tilde{\mathfrak{m}})$ satisfy Whitney's condition (a).

We now turn to Whitney's condition (b).
It has been noticed by Mather that (a) is implied by (b), but the converse is not always true. Instead, condition (a) implies condition (b), provided certain technical requirements hold. We now discuss these technicalities.
\begin{defin}
 Let $X$ and $Y$ be locally closed and disjoint submanifolds of $\RR^m$ with $Y \subset \cl(X)\backslash X$.
Let $\pi_Y:\nu_Y\to Y$ be a $C^\infty$-tubular neighborhood of $Y$ in $\RR^m$.
Then $(X,Y)$ is called $b^{\pi_Y}$-regular at $y \in Y$, if whenever
\begin{itemize}
    \item $\{x_k\}_{k=1}^{\infty}\subset X$ has a
limit $y\in Y$, 
\item the secant lines $L_k$ joining $x_k$ and $\pi_Y(x_k)$ converge to $L$, and 
\item $T_{x_k}X$ converges to a limit plane $P$,
\end{itemize}
 we can conclude that $L \subset P$.
\end{defin}

In the next proposition we will use distance tubes around a face to conclude that our faces satisfy Whitney's condition (b). 
 \begin{prop} \label{prop:conditionb}
Let $F(\m)$ be a singular face, which is contained
in the closure of a face $F(\tilde \m)$.  Then Whitney's condition {\rm (b)}
holds true for the pair $(F(\tilde \m),F(\m))$.
\end{prop}

\begin{proof}
The results of \cite{Tro2} imply that, since the pair $(F(\tilde \m),F(\m))$ satisfies condition (a) of Definition \ref{WTD}, it will also satisfy condition (b) if there is a distance tube
around $F(\m)$ which is $(b^{\pi_Y})$-regular. 
Let $D_\lambda \in F(\m)$ and $(S_k)_{k \in \NN}$ be a sequence in $F(\tilde \m)$ with
$S_k \to D_\lambda$ for $k\to \infty$. Let $T_k=\pi_{F(\m)}(S_k)$ denote the
orthogonal projection of $S_k$ onto $F(\m)$. Then the line $L_k$ through
$S_k$ and $T_k$  lies in the normal space $\nu_{T_k}F(\m)$. 
Moreover, for $S_k$ close to $D_\lambda$
the open segment through $S_k$ and $T_k$, not containing $T_k$,
lies in $F(\tilde \m)$ by Lemma \ref{lem:submanifold}. Indeed, if $T_k$ is diagonal then this is obvious
and for nondiagonal $T_k$ we may use the isometric action of ${\rm SO}(n)$ to diagonalize $T_k$.
It follows that $L_k \in T_{S_k}F(\tilde \m)$.
This shows the claim.
\end{proof}

We thus arrive at the first main result of this section.

\begin{thm}\label{thm:whitney}
The union $\Sing(n)$ of  singular faces is a closed, Whitney-stratified subspace of
$\Sym_0(n)$ of codimension two. The strata are the singular faces.
\end{thm}

\subsection{The singular set $\Sing(TM)\subset \Sym_0(TM)$}\label{subsec:SingM}

We now discuss the topological implications for the bundle of traceless and $g_0$-self-adjoint endomorphisms of the tangent bundle $TM$ of a manifold $(M,g_0)$.

\begin{defin} \label{def:SingM}
Let $(M,g_0)$ be a compact Riemannian manifold.
Then we define
$$
  \Sing(TM):=\{F \in \Sym_0(TM): F \textrm{ does not have pairwise distinct eigenvalues}\}\,.
$$
\end{defin}

\begin{lem}\label{lem:SingM}
Let $(M,g_0)$ be a compact Riemannian manifold of dimension \mbox{$n\ge 2$}.
Then $\Sing(TM)$ is a closed, Whitney-stratified subspace of $\Sym_0(TM)$ of codimension two.
\end{lem}

\begin{proof}
Let $p \in M$ and let $x:U_p\to V \subset \RR^n$ be a chart of $M$ around $p$ such that on $U_p$
there exists a $g_0$-orthonormal basis $e$.
Let $\pi:\Sym_0(TM)\to M;F_q \to q$ denote
the projection map. Then
we obtain a chart
$$
  \Phi_x :\pi^{-1}(U_p)\to V \times \Sym_0(n)\,\,;\,\,\,
  F_q \mapsto (x(q),M_{e_q}(F_q)) 
$$
of $\Sym_0(TM)$ where 
$M_{e_q}(F_q)$ denotes the $n\times n$ matrix 
representing
the endomorphism $F_q \in \End(T_q M)$,
 $q=\pi(F_q)$,
with respect to $e_q$. Of course
$$
  \Phi_x( \pi^{-1}(U_p)\cap \Sing(M))
   = V \times \Sing(n)\,.
$$
The claim follows now from Theorem \ref{thm:whitney}.
\end{proof}

We are now finally in a position to specify the directions of Ebin geodesics in which we are most interested in this paper. 
\begin{defin}[Generic Deformation]   \label{def:globalgen}
Let $(M,g_0)$ be a compact Riemannian manifold
of dimension $n\geq 2$ with volume form $\mu=\mu_{g_0}$. 
Then we call  $h \in T_{g_0}\Metmu$ \emph{generic},
if the corresponding endomorphism
$H :M \to \Sym_0(TM) $, defined in \eqref{eqn:H},
intersects  the singular set $\Sing(TM)$  transversally. 
\end{defin}

The set of all generic directions 
$h \in T_{g_0}\Metmu$
will be denoted by 
$\Gen$. Now we come to our second main result in this section.

\begin{thm}     \label{thm:generic}
Let $M$ be a compact manifold of dimension $n \geq 2$,
$g_0$ be a Riemannian metric on $M$ 
with volume form $\mu=\mu_{g_0}$.
Then, the set of generic directions $\Gen$ is an open and dense subset of $T_{g_0}\Metmu$ in the $C^\infty$-topology.
\end{thm}

\begin{proof}
This follows from \cite{Tro1} main Theorem on page 274
and Note (ii) on the same page.
\end{proof}

\subsection{The singular locus of a generic deformation}

In Definition \ref{def:SingM} we
defined the singular set
 $\Sing(TM)\subset \Sym_0(TM)$, which by 
 Lemma \ref{lem:SingM} is a closed, 
 Whitney-stratified subset. In this subsection,
 for generic $h$ we define the singular set $S_h \subset M$ as $h^{-1}(\Sing(TM))$, considering
 $h:M \to \Sym_0(TM)$ a smooth section.

\begin{defin}
  Let  $h:M \to S^2(T^*M)$ be smooth,
  $h \in T_{g_0}\Metmu$,
  and let  $H$ be
  the corresponding $g_0$-self-adjoint
  traceless endormorphism.
We say that at a point $p \in M$ the endomorphism $H$
has eigenvalue multiplicity  $\m=(m_1,...,m_L)$, if the eigenvalues 
$\lambda_1(p)\leq \cdots \leq \lambda_n(p)$ of $H(p)$ have the multiplicities $m_1,...,m_L$.
We set
$$
 \Str_{\m}(h):=\{ p \in M \mid H(p) \textrm{ has eigenvalue multiplicity } \m\} \subset M\,.
$$
For $\m=(m_1,...,m_L)$ we set
$\vert \m\vert :=L$
and call
$$
 S_h:=\bigcup_{\vert \m\vert <n } \Str_{\m}(h) 
 \quad \textrm{ and }\quad
 R_h:=M \backslash S_h 
$$
the \emph{singular}  and  \emph{regular locus} of $h$, respectively.
\end{defin}

We come to our second main result in this section.

\begin{thm}\label{thm:strat}
 Let $h \in T_{g_0}\Metmu$ be generic. 
Then the singular locus $S_h$ of $h$ is either empty or a compact Whitney stratified subspace of $M$
of codimension two.
\end{thm}

\begin{proof}
This is clear from the very definition
of generic deformations: the map
$H:M\to \Sym_0(TM)$ intersects
the singular set $\Sing(TM)$
transversally. Therefore, 
$S_h$ is either empty or a 
a compact Whitney stratified space of codimension two, the strata given
 by $\Str_m(h)$.
\end{proof}

In dimension $n=2$, for generic directions $h$ the singular locus 
consists of finitely many points, possibly none, whereas in dimension $n=3$
the singular locus consists of finitely many disjoint copies of $\mathbb{S}^1$. Starting with
dimension $n=4$, the multiplicity locus of an arbitrary generic direction will in general have singularities.

A particularly important aspect of the theory of generic directions is that not all multiplicities can occur. For example, if $n=4$ and $h$ is generic, then the multiplicity $\m_5^1=(3,1)$ cannot occur because a transverse 
intersection 
of the image of $H:M \to \Sym_0(TM)$
with $F(\m_5^1)$, which has co-dimension five, is of course impossible. 

Recall that our main goal in this paper is to apply Theorem \ref{curvatureasy} to show that scalar curvature typically converges exponential quickly to $-\infty$. 
We conclude this section with the observation that multiplicities which are `problematic', i.e., they obstruct the applicability of Theorem \ref{curvatureasy} too much, are absent starting in dimension six: See Lemma \ref{lem:submanifold}  concerning the codimension of a face $F(\m)$.

\begin{lem}\label{dimensionsixgood}
Let  $n\geq 6$ and
$\mathfrak{m}=(m_1,\cdots,m_L)$ be an eigenvalue multiplicity with 
$$ 
m_1=M_1,\,\,m_1+m_2=M_2,\,\,\cdots,\,\,n=M_L
$$
and $\sum_{i=1}^L d(m_i)\leq n$.
Let $\pi:\{1,\cdots,n\}\to \{1,\cdots,L\}$ be defined so that $\pi(i)$ is the unique number with $M_{\pi(i)-1} < i\le M_{\pi(i)}$ (with $M_0:=0$). 
Then,  there are at least $n+1$ triplets $(1,b,c)$ so that
 $1<b$ and   
$\pi(b)<\pi(c)$.
\end{lem}

\begin{proof}
First let us treat the $n=6$ case. Since $d(m_i)=\frac{m_i(m_i+1)}{2}-1$, it must be the case that $m_i<4$ for all $i$. Furthermore, if $m_i=3$ for any $i$, then the other multiplicities must all be $1$. Each  possible multiplicity is listed in the first column of the table below. For each multiplicity, the corresponding row provides a list of seven allowable triplets. 
\tiny
\begin{center}
\begin{tabular}{ |l|c|c|c|c|c|c|c| } 
 \hline
$\qquad\quad \mathfrak{m}$ & 1&2&3&4&5&6&7 \\
 \hline
 $\m_5^1$=(3,1,1,1) & (1,2,4)&(1,2,5)&(1,2,6)&(1,3,4)&(1,3,5)&(1,3,6)&(1,4,5) \\ 
 $\m_5^2$=(1,3,1,1) & (1,2,5)&(1,2,6)&(1,3,5)&(1,3,6)&(1,4,5)&(1,4,6)&(1,5,6)  \\ 
 $\m_5^3$=(1,1,3,1) & (1,2,3)&(1,2,4)&(1,2,5)&(1,2,6)&(1,3,6)&(1,4,6)&(1,5,6)  \\
 $ \m_5^4$=(1,1,1,3) & (1,2,3)&(1,2,4)&(1,2,5)&(1,2,6)&(1,3,4)&(1,3,5)&(1,3,6)\\
 $ \m_6$=(2,2,2) & (1,2,3)&(1,2,4)&(1,2,5)&(1,2,6)&(1,3,5)&(1,3,6)&(1,4,5)\\
 $\m_4^1$=(2,2,1,1) & (1,2,3)&(1,2,4)&(1,2,5)&(1,2,6)&(1,3,5)&(1,3,6)&(1,4,5)\\
  $ \m_4^2$=(2,1,2,1) & (1,2,3)&(1,2,4)&(1,2,5)&(1,2,6)&(1,3,5)&(1,3,6)&(1,4,6)\\
  $ \m_4^3$= (2,1,1,2) & (1,2,3)&(1,2,4)&(1,2,5)&(1,2,6)&(1,3,5)&(1,3,6)&(1,4,5)\\
 $ \m_4^4$=(1,2,2,1) & (1,2,4)&(1,2,5)&(1,2,6)&(1,3,4)&(1,3,5)&(1,3,6)&(1,4,6)\\
  $\m_4^5$=(1,2,1,2) & (1,2,4)&(1,2,5)&(1,2,6)&(1,3,4)&(1,3,5)&(1,3,6)&(1,4,6)\\
 $\m_4^6$=(1,1,2,2) & (1,2,3)&(1,2,4)&(1,2,5)&(1,2,6)&(1,3,5)&(1,3,6)&(1,4,6)\\
 $\m_2^1 $= (2,1,1,1,1) & (1,2,3)&(1,2,4)&(1,2,5)&(1,2,6)&(1,3,4)&(1,3,5)&(1,3,6) \\ 
 $\m_2^2$=(1,2,1,1,1) & (1,2,4)&(1,2,5)&(1,2,6)&(1,3,4)&(1,3,5)&(1,3,6)&(1,4,5)  \\ 
 $ \m_2^3$=(1,1,2,1,1) & (1,2,3)&(1,2,4)&(1,2,5)&(1,2,6)&(1,3,5)&(1,3,6)&(1,4,5)  \\
 $\m_2^4$=(1,1,1,2,1) & (1,2,3)&(1,2,4)&(1,2,5)&(1,2,6)&(1,3,4)&(1,3,5)&(1,3,6)\\
 $ \m_2^5$=(1,1,1,1,2) & (1,2,3)&(1,2,4)&(1,2,5)&(1,2,6)&(1,3,4)&(1,3,5)&(1,3,6)\\
$ \m_0$= (1,1,1,1,1,1) & (1,2,3)&(1,2,4)&(1,2,5)&(1,2,6)&(1,3,4)&(1,3,5)&(1,3,6)\\
 \hline
\end{tabular}
\end{center}
\normalsize
We proceed by induction, using the $n=6$ case as the base step. 
If $n\ge 7$, then for any multiplicity $\mathfrak{m}=(m_1,\cdots,m_L)$ with
$\sum_{i=1}^L d(m_i)\leq n$, we proceed in cases:
\begin{itemize}
    \item If $m_i=1$ for all $i>1$, then $L-1=n-m_1$, and there are $m_1\cdot \frac{(L-1)(L-2)}{2}$ allowable triplets.  For $m_1=2,3$, it is clear that this expression is at least $n+1$. We now suppose that $m_1\ge 4$ and $d(m_1)=\frac{m_1(m_1+1)}{2}-1\le n$. We conclude that $m_1\le \frac{1}{2}\sqrt{8n+9}-\frac{1}{2}$, so 
    \begin{align*}
        m_1\cdot \tfrac{(L-1)(L-2)}{2}
        \ge 
        2 \big( n+\tfrac{1}{2}-\tfrac{1}{2}\sqrt{8n+9}\big)
        \cdot 
        \big( n-\tfrac{1}{2}-\tfrac{1}{2}\sqrt{8n+9}\big);
    \end{align*}
    this expression is at least $n+1$ whenever $n\ge 7$. 
    \item If $m_i>1$ for some $i>1$, then the multiplicity $\tilde{\mathfrak{m}}=(m_1,\cdots,m_{i-1},m_i-1,m_{i+1},\cdots,m_L)$ satisfies the required estimate for dimension $n-1$, and so there are at least $n$ allowable triplets. The result follows since $\mathfrak{m}$ clearly has at least one more allowable triplet. 
\end{itemize}

\end{proof}

We now turn to dimension five.

\begin{lem}\label{tabledim5}
Let  $n= 5$ and
$\mathfrak{m}=(m_1,\cdots,m_L)$ be an eigenvalue multiplicity with 
$$ 
m_1=M_1,\,\,m_1+m_2=M_2,\,\,\cdots,\,\,n=M_L
$$
and $\sum_{i=1}^L d(m_i)\leq 5$.
Then, all the triplets $(1,b,c)$ with
 $1<b$ and   
$\pi(b)<\pi(c)$ are listed in the following table:
\tiny
{\rm 
\begin{center}
\begin{tabular}{ |l|c|c|c|c|c|c| } 
 \hline
 $\qquad\quad\mathfrak{m}$ & 1&2&3&4&5&6\\
 \hline
 $\m_5^3$=(1,1,3) & (1,2,3)&(1,2,4)&(1,2,5)&-&-&-\\
 $\m_5^2$=(1,3,1) & (1,2,5)&(1,3,5)&(1,4,5)&-&-&-\\
 $\m_5^1$=(3,1,1) & (1,2,4)&(1,2,5)&(1,3,4)&(1,3,5)&-&-\\ 
  $\m_4^3$=(1,2,2) & (1,2,4)&(1,2,5)&(1,3,4)&(1,3,5)&-&-\\
  $\m_4^2$=(2,1,2) & (1,2,3)&(1,2,4)&(1,2,5)&(1,3,4)&(1,3,5)&-\\
   $\m_4^1$=(2,2,1) & (1,2,3)&(1,2,4)&(1,2,5)&(1,3,5)&(1,4,5)&-\\
   $\m_2^4$= (1,1,1,2) & (1,2,3)&(1,2,4)&(1,2,5)&(1,3,4)&(1,3,5)&-\\
    $\m_2^3$= (1,1,2,1) & (1,2,3)&(1,2,4)&(1,2,5)&(1,3,5)&(1,4,5)&-\\
       $\m_2^2$=(1,2,1,1) & (1,2,4)&(1,2,5)&(1,3,4)&(1,3,5)&(1,4,5)&-\\
       $\m_2^1$=(2,1,1,1) & (1,2,3)&(1,2,4)&(1,2,5)&(1,3,4)&(1,3,5)&(1,4,5)\\
      $\m_0$= (1,1,1,1,1) & (1,2,3)&(1,2,4)&(1,2,5)&(1,3,4)&(1,3,5)&(1,4,5)\\
 \hline
\end{tabular}
\end{center}
}
\normalsize 
\end{lem}

Even though we do not have
sufficiently many nice triples, we notice
that the number of the triples in each
case still exceeds the dimension
of the corresponding stratum by at least one.

We now turn to dimension four.

\begin{lem}
Let  $n=4$ and
$\mathfrak{m}=(m_1,\cdots,m_L)$ be an eigenvalue multiplicity with 
$$ 
m_1=M_1,\,\,m_1+m_2=M_2,\,\,\cdots,\,\,n=M_L
$$
and $\sum_{i=1}^L d(m_i)\leq 4$.
Then, all the triplets $(1,b,c)$ with
 $1<b$ and   
$\pi(b)<\pi(c)$ are listed in the following table:
\tiny
{\rm 
\begin{center}
\begin{tabular}{ |l|c|c|c| } 
 \hline
 $\qquad\quad\mathfrak{m}$ & 1&2&3\\
 \hline
 $\m_4$=(2,2) & (1,2,3) & (1,2,4) & - \\
   $\m_2^3$= (1,1,2) & (1,2,3)&(1,2,4)&-\\
    $\m_2^2$= (1,2,1) & (1,2,4)&(1,3,4)&-\\
       $\m_2^1$=(2,1,1) & (1,2,3)&(1,2,4)&(1,3,4)\\
      $\m_0$= (1,1,1,1) & (1,2,3)&(1,2,4)&(1,3,4)\\
 \hline
\end{tabular}
\end{center}
}
\normalsize 
\end{lem}

We see that there can be two $2$-dimensional strata in $M$ for which
we have only $2$ triples. Even worse,
on the regular part, we have only $3$ triples.

\begin{rem}\label{rem:dim4} Recall that in Lemma \ref{lem:3dplus}, we showed that the set of directions in which scalar curvature tends uniformly to $-\infty$ cannot be dense. The proof involved pointing out that in dimension three, according to Corollary \ref{cor:scalfor}, there are only three scalar functions that could possibly force scalar curvature exponentially towards $-\infty$, and that these three functions could be chosen to have a transversal intersection with $0$. Such a proof cannot be extended to dimension four, because there are now six functions (the $3$ functions coming from the $3$ triples $(1,2,3)$, $(1,2,4)$ and $(1,3,4)$, and also the $3$ functions  $e_1\lambda_i$, $i=1,2,3$). On the other hand, we have not established generic convergence of scalar curvature to $-\infty$ either, because we do not
have such a nice formula as in the regular case to treat the scalar curvature on the $2$-dimensional strata for $\mathfrak{m}_2^3=(1,1,2)$ and $\mathfrak{m}_2^2=(1,2,1)$. 
\end{rem}


\section{Perturbation}
We now turn attention to proving Theorem \ref{thm0}. Recall that our candidate for the open and dense set is $\mathcal{Y}_{g_0}$ from  Definition \ref{candidateopendense}. We already know that this set is open (Lemma \ref{lem:candidateopen}), so we need to show density. Recall
that we denoted by $\Gen$ the set of generic directions $h \in T_{g_0}\Metmu$.
To start, we observe that it is possible to ensure good eigenvalue clustering by shrinking neighbourhoods.
\begin{lem}\label{nbhdeigencluster}
    Let $(M,g_0,p)$ be a pointed compact Riemannian manifold of dimension $n\ge 5$, and let $h_0\in \Gen$. Let the corresponding traceless endomorphism $H_0$ have eigenvalue multiplicity $\mathfrak{m}(p)=(m_1,\cdots,m_L)$ at the point $p$. Then there exists a compact neighbourhood $K_p$ of $p$ in $M$ (with smooth boundary), and an open neighbourhood $\mathcal{Z}_{p}$ of $h_0$ in $\Gen$ so that
    for each $h\in \mathcal{Z}_{p}$, there exists a smooth $g_0$-orthonormal frame $e_h$ on $K_p$ (depending smoothly on $h$) in which the corresponding $H$ appears in the block diagonal form \eqref{blockdiagonalH} with multiplicity $\mathfrak{m}(p)$ and strictly increasing smooth eigenvalue functions $\lambda_1 < \cdots< \lambda_L:K_p \to \RR$, and $$
     \Vert S\Vert \leq \tfrac{1}{12000}\cdot \min_{K_p}\{\left|\lambda_j-\lambda_k\right|, j \neq k\}.
     $$
\end{lem}
\begin{proof}
Since $n\ge 5$ and $h_0$ is generic, we must have $L\geq 2$: let the eigenvalues of $H_0$ at $p$ be $\lambda_1^*< \lambda_2^* < \cdots < \lambda_L^*$. Let $r:=\min\{\lambda_{i+1}^*-\lambda_i^*: i= 1,\cdots, L-1\}$ and 
$\epsilon=\tfrac{r}{12000\cdot (C_2(n)+\tilde{C_2}(n))}$. Then
there is a compact neighbourhood $K_p$ of $p$ on which there are $m_i$ eigenvalues of $H_0$ in $(\lambda_i^*-\tfrac{\epsilon}{2},\lambda_i^*+\tfrac{\epsilon}{2})$. Moreover,
there is an open neighbourhood $\mathcal{Z}_{p}$ of $h_0$ in $\mathcal{G}_{g_0}$ so that, for each $h\in \mathcal{Z}_{p}$, the eigenvalues of $H$
on all of $K_p$ are in $(\lambda_i^*-\epsilon,\lambda_i^*+\epsilon)$. By Lemma \ref{lem:goodframe} and Lemma \ref{lem:smoothframe}, the claim follows; we may assume that 
$K_p =x^{-1}(\overline{B_{1}(0)})$ as in Lemma
\ref{lem:goodframe}.   
\end{proof}

We now describe how to perturb tensors $h\in \mathcal{Z}_{p}$ so that the resulting frames satisfy \eqref{eqn:posstr} locally.

\begin{prop}\label{IFT}
Let $(M,g_0,p)$ be a pointed compact Riemannian manifold of dimension $n\ge 6$, and let $h_0\in \Gen$. Let the corresponding traceless endomorphism $H_0$ have eigenvalue multiplicity $\mathfrak{m}(p)=(m_1,\cdots,m_L)$ at the point $p$. 
Let $K_p=x^{-1}(\overline{B_{1}(0)})\subset x^{-1}(B_4(0))$ be the compact neighbourhood
of $p$ in $M$ 
and $\mathcal{Z}_{p}$ be an open neighbourhood 
of $h_0$ in $\Gen$ satisfying the conclusion of Lemma \ref{nbhdeigencluster}. Let $J$ be a collection of pairs $(j,k)$ with $1<j<k$. If $\left|J\right|>n$, then it is possible to shrink
$\mathcal{Z}_{p}$ to an open set $\mathcal{Y}_p$ and shrink the associated compact set $K_p$ so that $\mathcal{D}_p$ is dense in $\mathcal{Y}_p$ (in the $C^{\infty}$ topology). Here, $\mathcal{D}_p\subseteq \mathcal{Y}_p$ is  
the set directions for which the corresponding $g_0$-orthonormal frames $\{e_i\}_{i=1}^{n}$ (from Lemma \ref{nbhdeigencluster}) satisfy the following property: 
at every point on 
$K_p:=x^{-1}(\overline{B_{1}(0)})$,
at least one of the functions $\{g_0([e_1,e_j],e_k)\}_{(j,k) \in J}$ 
is non-zero.
\end{prop}

\begin{proof}
We want to apply the Thom Transversality Theorem, specifically, Theorem
4.9 and Corollary 4.10 on page 56 in \cite{GG}.
  For a given co-ordinate system $(U,x)$ of $M$, 
  we denote by $G$ the coefficient
matrix of $g$ with respect
the chart $x$,
given by
$G_{ij}=g_0(\tfrac{\partial}{\partial x_i},\tfrac{\partial}{\partial x_j})$ for $1 \leq i,j \leq n$. 
By adjusting co-ordinates, we may assume that $G(p)=I_n$, $p=x^{-1}(0)$, and that on $U=x^{-1}(B_4(0))$ we have
a $g_0$-orthonormal frame $e_0=(e_1^0,\cdots ,e_n^0)$
with $e_i^0(p)=\tfrac{\partial}{\partial x_i}(p)$
for all $i=1,\cdots,n$.
For each point $q \in U$ and each $A \in C^{\infty}(U;\SO(n))$, we have 
$$
\sqrt{G^{-1}(q)}A(q)\sqrt{G(q)} \in \SO(T_qM,g_q),
$$
so that the change of basis matrix function $\sqrt{G^{-1}}A\sqrt{G}$ applied to the frame $\{e_i^0\}_{i=1}^{n}$ gives us another $g_0$-orthonormal frame.
Next,
using the chart $x$ we consider $G$, $\{e_i^0\}_{i=1}^n$ and $g_0$
to be functions of the variable $x\in B_2(0)$.
Moreover, let 
$$
 V:B_2(0)\to \hat V
:=\mathfrak{so}(n)=T_e \SO(n)=\RR^{\frac{1}{2}n(n-1)}
$$
be a smooth function.
We consider the function $F^V:B_2(0)\to \RR^N$, $N=\vert J\vert$, with
coordinate functions $F_{jk}^V:B_2(0)\to \RR$
given by
$$
F_{jk}^V:=g_0([e_1^V,e_j^V],e_k^V)
$$ 
where for all $j=1,\cdots, n$
$$
e_j^V:=\sqrt{G^{-1}}\text{exp}(V)\sqrt{G}\cdot e_j^0=\sum_{l=1}^n (\sqrt{G^{-1}}\text{exp}(V)\sqrt{G})_{lj}\cdot e_l^0.
$$
Note that for the zero function $V\equiv 0$
we have $F^V_{jk}=g_0([e_1^0,e_j^0],e_k^0)$.

Let $\hat W:={\rm Hom}(\RR^n, \hat V)$
and consider the smooth function 
\begin{align}\label{eqn:defF}
    \hat F:B_2(0)\times \hat V \times \hat W\to \RR^N\,\,;\,\,\,(x,V,W) \mapsto \hat F(x,V,W)
\end{align}
with coordinate functions
$\hat F_{jk}$,  $(j,k)$ ranging over $J$,
which satisfies
$$
    F^V(x)=\hat F(x,V(x),(DV)_x)
$$ 
for all $x \in B_2(0)$. 
To calculate $\hat{F}_{jk}$, we let 
$$
P(V,W)=\tfrac{d}{dt}\vert_{t=0} \exp(V+tW)\in \text{Hom}(\mathbb{R}^n,\exp(V)\hat{V})
$$
and recall that for smooth vector fields $X,Y$ on 
open subsets of $\RR^n$ we have $[X,Y]=DY\cdot X - DX\cdot Y$.
Thus, 
we obtain
\begin{align*}
    \hat{F}_{jk}(x,V,W)
    &=
    g_0(\sqrt{G^{-1}}\text{exp}(V)P(V,W)(e_1^V)\sqrt{G}\cdot e_j^0,e_k^V)\\
    &
    -g_0(\sqrt{G^{-1}}\text{exp}(V)P(V,W)(e_j^V)\sqrt{G}\cdot e_1^0,e_k^V)\\
    &+\hat{H}_{jk}(x,V),
\end{align*}
where $\hat{H}_{jk}$ captures the terms that do not involve any differentiation of $V$.
Now, since $P(0,W)=W$ and $G(0)=I_n$, the linearisation of $\hat{F}_{jk}$ at $(0,0,W_0)$ is given by 
\begin{align*}
    D(\hat{F}_{jk})_{(0,0,W_0)}(y,V,W)
    =
    g_0(W(e_1^0)\cdot e_j^0-W(e_j^0)\cdot e_1^0,e_k^0)_0
    +\hat{\tilde{H}}_{jk}(y,V),
\end{align*}
where $\hat{\tilde{H}}_{jk}$ is the linearisation of $\hat{H}_{jk}$ at $(0,0)$. By examining the $W$ components, it is clear that this linearisation has full rank. 
Thus $\hat{F}_{jk}^{-1}(0)$ defines a co-dimension one submanifold in the jet space. 
In fact, by examining the $W$ component for all $(j,k)\in J$, we find that these co-dimension one submanifolds $S_{jk}$ intersect transversally for $x\in B_{4\epsilon}(0)$ for
some $\epsilon>0$ and $V\in B_\delta(0)$
for some $\delta>0$ because $g_0(W(e_1)\cdot e_j,e_k)$ are independent in the jet space for $j<k$.  Thus,
$\hat N:=\hat F^{-1}(0)$ is a smooth submanifold in the
neighbourhood
$B_{4\epsilon}(0)\times B_\delta(0)\times \hat W$
of $(0,0,0)$ in 
$ B_2(0)\times \hat V\times \hat W$ of codimension $N=\Vert J\Vert >n$.
We now replace $B_4(0)$ by $B_{4\epsilon}(0)$ and
choose an open neighbourhood $\mathcal{Y}_p'$ of $h_0$
in $\Gen$ such that for all $h \in \mathcal{Y}_p'$
we have a block diagonalizing $g_0$-orthonormal basis
$e_h$ given by $V:B_2(0)\to \hat V$ with
$V(x) \in B_\delta(0)$ for all $x \in B_2(0)$.
By the Thom Transversality theorem the claim follows.
\end{proof}

\begin{prop}\label{prop:genericperturbation>6}
Let $(M,g_0,p)$ be a pointed compact Riemannian manifold of dimension $n\ge 6$, and let $h_0\in \Gen$. Let the corresponding traceless endomorphism $H_0$ have eigenvalue multiplicity $\mathfrak{m}(p)=(m_1,\cdots,m_L)$ at the point $p$. It is possible to find a compact neighbourhood $K_p$ with interior point $p$ and open neighbourhood $\mathcal{Y}_p$ of $h_0$ so that each $h\in \mathcal{Y}_p$ can be approximated by $h_i \in \Gen$ (on all of $M$)  such that \eqref{eqn:posstr} 
    holds at all points in $K_p$.
\end{prop}

\begin{proof}
Let $e_0:=e_{h_0}$ be a $g_0$-orthogonal frame for $h_0$
giving $H_0$ block diagonal form on $x^{-1}(B_2(0))$.
By eventually shrinking $U=x^{-1}(B_4(0))$ we may assume in addition to satisfying the conclusion of Lemma \ref{nbhdeigencluster}. In addition, by composing $(dx)_p:T_pM \to \RR^n$ with a linear transformation
of $\RR^n$, we can assume that 
that $(dx)_p \cdot e_0^i(p)=e_i^s$
for all $i=1,\cdots, n$, $(e_1^s,...,e_n^s)$
standard basis of $\RR^n$.

The result follows from the fact that $H_0$ is generic. Indeed, since $n\ge 6$, Lemma \ref{dimensionsixgood} implies that there are at least 
$N:=\vert J\vert >n$ functions $(g_0)^{1_a,j_b,k_c}$ for which the triplets $(1_1,j_b,k_c)$ satisfy $1\leq j<k \leq L$,  $b\le m_j$ and $c\le m_k$. We let $J$ denote the set of all such pairs $(j_b,k_c)$.
Therefore, by Proposition \ref{IFT}, it is possible to find an open $\mathcal{Y}_p$ so that for each $h\in \mathcal{Y}_{p}$, it is possible to approximate the corresponding $g_0$-orthonormal frame $e_h$ by a sequence $e_{h_i}$ on $K_p$ so that 
\eqref{eqn:posstr} holds on all of
$K_p$. 
Recall that the frames $e_{h_i}$ are parametrized
by maps $V_i:x^{-1}(B_2(0))\to \so(n)$ and
that $e_h$ corresponds to a choice of $V_h$ which is small in the $C^{\infty}$ topology. 
We may and will
assume that $V_i$ converges to $V_h$ on
$x^{-1}(B_2(0))$ in $C^l$ topology for every
$l\in \NN$.

Let $\psi:M\to [0,1]$
be a smooth cutoff function with
$\psi\vert_{x^{-1}(\overline{B_1(0))}}\equiv 1$
and $\psi(q)=0 $ for $q \not \in 
x^{-1}(B_{1.5}(0))$.
 We consider
the frame $e_{i}$ corresponding to $\psi \cdot V_i+(1-\psi)V_h$, defined on $x^{-1}(B_2(0))$. 
It follows that $e_{i}$ converges to $e_h$
in $C^l$ topology on $x^{-1}(B_2(0))$. Next,
we define $H_i$ to be $H$ outside $x^{-1}(B_{2}(0))$, and inside, $H_i$ has the same eigenvalue functions and off-diagonal terms as
$H$, but with eigenbasis $e_{i}$ (instead of $e_h$). 
It follows that $H_i$ converges to $H$ in
the $C^l$ topology everywhere. 
Now for each $l \in \NN$
let $\tilde H_l$ be taken from the
above sequence with $\Vert \tilde H_l-H\Vert_{C^l}< \tfrac{1}{l}$.
Then $(\tilde H_l)_{l \in \NN}$
converges to $H$ in $C^\infty$ topology.
This shows the claim.
\end{proof}


\begin{proof}[Proof of Theorem \ref{thm0} in  case dimension $n\ge 6$]
We already know 
by Lemma \ref{lem:candidateopen} that 
$\mathcal{Y}_{g_0}$ is open in $\Gen$, and since $\Gen$ is open and dense in $T_{g_0}\mathcal{N}_{\mu}$, it suffices to show that $\mathcal{Y}_{g_0}$ is dense in $\Gen$ (this shows in particular
that $\mathcal{Y}_{g_0}$ is non-empty).

To that end, choose $h_0\in \Gen$. 
 Proposition \ref{prop:genericperturbation>6} implies, for each $p\in M$, the existence of a compact subset $K_p$, with $p$ as an interior point, a neighbourhood $\mathcal{Y}_p$ of $h_0$ in $\Gen$, such that each $h \in \mathcal{Y}_p$
 can be approximated by $h_i^p$ in $C^\infty$ topology, such that \eqref{eqn:posstr} holds on all of $K_p$ for
 these endomorphisms.
 
 We now take a finite subcover $M=\cup_{i=1}^P K_i$ given by points $p_1,\cdots,p_P$. 
 We set $ \mathcal{\tilde Y}_{h_0}:=\cap_{i=1}^P \mathcal{Y}_{p_i}$, still an open neighbourhood
 of $h_0$ in $\Gen$.
 By Proposition
 \ref{prop:genericperturbation>6} we find $H_1 \in 
 \mathcal{ \tilde Y}_{h_0}$ such that on $K_1$ we have
 for \eqref{eqn:posstr} a lower
 bound  $ \delta_1 > 0$. 
 Inductively, we may assume that
 there exists $H_l \in 
 \mathcal{ \tilde Y}_{h_0}$ such that on $\cup_{i=1}^l K_i$
 we have  for \eqref{eqn:posstr} a lower bound 
 $ \delta_l > 0$. 
 The induction step follows again
 from Proposition
 \ref{prop:genericperturbation>6},
 since $H_l$ can be approximated by
 $H_{l+1}$ (on all of $M$) in $C^\infty$
 topology  such that on all of $K_{l+1}$
 \eqref{eqn:posstr} holds.
 Since $H_{l+1}$ can be chosen as close 
 to $H_l$ as we like,
 for $H_{l+1}$ we still may assume that 
 on  $\cup_{i=1}^l K_i$
we have  for \eqref{eqn:posstr} a lower bound 
 $
 \tfrac{\delta_l}{2} > 0$. 
 This shows
that $\mathcal{Y}_{g_0}$ is non-empty.
Clearly, the above argument shows that
$\mathcal{Y}_{g_0}$ is also dense.
\end{proof}

\begin{prop}\label{genericperturbation>5}
Let $(M,g_0,p)$ be a pointed compact Riemannian manifold of dimension $n=5$, and let $h_0\in \Gen$. Let the corresponding traceless endomorphism $H_0$ have eigenvalue multiplicity $\mathfrak{m}(p)=(m_1,\cdots,m_L)$ at the point $p$. Then there exists a compact neighbourhood $K_p$ of $p$ in $M$ (with smooth boundary), and an open neighbourhood $\mathcal{Y}_{h_0}$ of $h_0$ in $\Gen$ with the following properties:
\begin{itemize}
    \item[{\rm (1)}] For each $h\in \mathcal{Y}_{h_0}$ there exists a smooth $g_0$-orthonormal frame $e_h$ on $K_p$ (depending smoothly on $h$) in which the corresponding $H$ appears in the block diagonal form \eqref{blockdiagonalH} with multiplicity $\mathfrak{m}(p)$ and strictly increasing smooth eigenvalue functions $\lambda_1 < \cdots< \lambda_L:K_p \to \RR$, and $$
     \Vert S\Vert \leq \tfrac{1}{12000}\cdot \min_{K_p}\{\left|\lambda_j-\lambda_k\right|, j \neq k\}.
     $$
    \item[{\rm (2)}] Each $h\in \mathcal{Y}_{h_0}$ can be approximated by $h_i \in \Gen$ (on all of $M$) so that there is some $j\le L$, $b\le m_j$, $j<k\le L$, $c\le m_k$ so that $(g_0)^{1_1,j_b,k_c}\neq 0$ at all points in $K_p\cap Str_{\mathfrak{m}(p)}(h)$. 
\end{itemize}
\end{prop}

\begin{proof}
The first point is Lemma 
\ref{nbhdeigencluster}. The second point follows similarly to Proposition \ref{prop:genericperturbation>6}.
Notice, that we have less triplets $(1_1,j_b,k_c)$ with $1\le j<k$, $b\le m_j$ and $c\le m_k$: see  Lemma \ref{tabledim5}. However, the number of allowable triplets plus the co-dimension of $Str_{\mathfrak{m}(p)}(h_0)$ is 
strictly larger than five, thus we are in position
to repeat the proof of Proposition \ref{IFT}.
We thus find that by perturbing the frame, taking care not to perturb the $\lambda_i$'s and $S_i$'s (so that $Str_{\mathfrak{m}(p)}(h_0)=Str_{\mathfrak{m}(p)}(h)$), it is possible to ensure that at least one relevant function is non-zero on $K_p\cap Str_{\mathfrak{m}(p)}(h)$. 
\end{proof}

We would like to mention that the strata of lowest
dimension are compact. However, strata of higher dimension,
containing a stratum of lower dimension in its closure,
will not be compact anymore.

\begin{proof}[Proof of Theorem \ref{thm0} in dimension $5$]
As in the proof for dimension $n\ge 6$, it will suffice to show that $\mathcal{Y}_{g_0}$ is non-empty, and dense in $\Gen$.

Let $h_0 \in \Gen$ be generic.
Recall that by Lemma \ref{tabledim5},
for the codimenion $5,4,2,0$ strata of $h_0$ we have at least
$3,4,5,6$ functions, respectively. Similar to the proof for
dimension $n\geq 6$ we consider a compact covering $K_5$
of the (compact) codimension $5$ strata $S_5 \subset M$ with
$S_5 \subset {\rm int}(K_5)$
on which we can change $h_0$ (with Propsition \ref{genericperturbation>5}, so we only change the eigenbasis) 
into $h_5$ such that on all of $K_5$ we have
always at least one non-vanishing function being strictly bigger that $\delta_5>0$. Next, we consider $M_4:=\overline{M\backslash K_5}$. Notice that the intersection $S_4$ of the
codimension 4 strata of
$h_0$ (or $h_5$) with $M_4$ is compact. We cover again $S_4$ by compact $K_4 \subset M$ with $S_4 \subset  {\rm int}(K_4)$
and change $h_5$ into $h_4$ (using Proposition \ref{genericperturbation>5}) such that on all of $K_4$
we have at least one non-vanishing function being strictly bigger that $\delta_4>0$. By choosing
$h_4$ as close to $h_5$ (on all of $M$) 
as we like we guarantee that
on $K_5$ the above non-vanishing function
is still strictly bigger than $\tfrac{\delta_5}{2}$.
We proceed precisely in the same manner for the codimension
2 strata and the generic codimension 0 stratum. This
shows that $\mathcal{Y}_{g_0}$ is non-empty.
Again, as in the proof of Theorem \ref{thm0} above,
it follows easily that $\mathcal{Y}_{g_0}$ is dense.
\end{proof}

\section{The normal form}
\label{sec:gener-deform-norm}

In this section, we prove Theorem \ref{thm:normal-form} below, which gives a convenient normal form for the tensor $H$. This theorem is, in a sense, the most refined version of Lemma \ref{lem:goodframe} we can possibly produce in case the tensor $H$ is generic. Although we do not pursue applications of this result, we expect this special form will prove useful in further understanding scalar curvature asymptotics.    
\begin{thm}\label{thm:normal-form}
Define $d:\mathbb{N}\to \mathbb{N}\cup\{0\}$ with $d(m_i)=\tfrac{1}{2}m_i(m_i+1)-1$. 
Let $h \in T_{g_0}\Metmu$ be generic and let $p \in \Str_{\m}(h)$ be a point on the multiplicity locus of $h$
for $\m=(m_1,...,m_L)$. Suppose that $m_1,...,m_r\geq 2$ and that
$m_{r+1}=\cdots = m_L=1$ for some $r\in \{1,...,L\}$.
Let 
$$
 d_\m:=\dim \Str_{\m}(h)=n-\codim(F(\mathfrak{m}))=n -\sum_{i=1}^L
 d(m_i).
 $$
Then there exists a ball $B_p$ around $p$, a smooth $g_0$-orthonormal 
frame field $e$  and coordinates $x=(a,b_1,...,b_r)$  on $B_p$, 
$a \in \RR^{d_\m}$, 
$b_i=(b_i^1,...,b_i^{d(m_i)})\in \RR^{d(m_i)}$ for $i=1,...,r$, 
such that
     \begin{eqnarray*}
    H(x) &=&
    \begin{pmatrix}
      {\bar \lambda}_1(x) I_{m_1}+S_1(b_1) & & &&&\\
      & \ddots & && &\\
       & & {\bar \lambda}_r(x) I_{m_r}+S_r(b_r) &&&\\
      &&&\bar \lambda_{r+1}(x)&& \\
       &&&& \ddots & \\
       &&&&&  \bar \lambda_{L}(x)
    \end{pmatrix} \,.
  \end{eqnarray*}
Here $\bar \lambda_j:B_p \to \RR$ are
smooth functions, $j=1,...,L$, and 
  \begin{equation*}
    S_i(b_i) :=
    \begin{pmatrix}
      b_i^1 & b_i^2 & \cdots & b_i^{m_i - 1} & b_i^{m_i} \\
      b_i^2 & b_i^{m_i + 1} & \cdots & b_i^{2m_i - 2} & b_i^{2m_i - 1} \\
      \vdots & \vdots & \ddots & \vdots & \vdots \\
      b_i^{m_i - 1} & b_i^{2m_i - 2} & \cdots & b_i^{d(m_i) - 1} &   b_i^{d(m_i)}\\
      b_i^{m_i} & b_i^{2m_i - 1} & \cdots & b_i^{d(m_i)} & \bar z(b_i)
    \end{pmatrix}\,,
  \end{equation*}
 with $\bar z(b_i)= -(b_i^1 + b_i^{m_i + 1} + \cdots + b_i^{d(m_i) - 1})$
for $i=1,...,r$. 
\end{thm}
\begin{proof}
Let $h \in T_{g_0}\Metmu$ be a generic direction and $p \in \Str_{\m}(h)$
be a point on the multiplicity locus of $h$. 
Locally, the normal bundle of the $d_\m$-dimensional submanifold 
$\Str_{\m}(h) \subset M^n$ is trivial, that is there exists an open neighborhood $B=B_p$ of $p$, which is diffeomorphic
to $D^m \times D^{n-m}$, where the first factor $D^m$ corresponds to $\Str_{\m}(h)$,
that is $m=d_\m$. Suppose that the origin $(0,0)$ corresponds to the point $p$.

Let $e$ be now a $g_0$-orthonormal frame on $B_p$ as in Lemma \ref{lem:goodframe}.
Then, the image of $H=H_B^e$ has the nice form described in Lemma \ref{lem:goodframe}.
We set $\bar \lambda_j(x):=\frac{1}{m_i}\cdot \tr H_j(x)$ for
$j=1,...,L$ and
obtain  $H_i =\bar \lambda_i \cdot I_{m_i} +S_i$, where $\tr S_i=0$, for $i=1,...,r$.
Let now the diagonal $(n\times n)$-matrix $D_{\bar \lambda(x)}$ be given by  
the diagonal entries $(\bar \lambda_1(x),...,\bar\lambda_L(x))$
with multiplicity $\m=(m_1,...,m_L)$ and let
 $S(x) \in \nu_{D_{\bar \lambda(x)}}F(\m)$ be given 
by $S_1(x),...,S_r(x)$ (see Lemma \ref{lem:submanifold}).
Let  $(a,b)$ denote the standard coordinates on    $D^m \times D^{n-m}$ and set
 $ \Phi(a,b):=(a, S(a,b))$. We have
 \begin{eqnarray*}
  (D\Phi)_{(a,b)} &=& \left( \begin{array}{cc} I_m & * \\    0 & \tfrac{\partial S(a,b)}{\partial b}\end{array}
    \right)\,.
\end{eqnarray*}
Since the image of $H$ intersects the  face $F(\m)$ transversally, we conclude
that $(D\Phi)_{(a,0)}$ is an isomorphism for all $a \in D^m$. As a consequence, on 
$D^m \times D^{n-m}$ there exist coordinates $(\tilde a,\tilde b)$
such that $\Phi(\tilde a,\tilde b) =(\tilde a,\tilde b)$, thus $S(\tilde a,\tilde b)=\tilde b$.
This implies the claim.
\end{proof}




\end{document}